\documentclass[a4paper,12pt,oneside,reqno]{amsart}
\usepackage[ansinew]{inputenc}
\usepackage{xcolor}
\usepackage{amsfonts}
\usepackage{amssymb}
\usepackage{amsmath,amsthm,amsfonts,amssymb}
\usepackage[left=3.25cm,right=3.25cm,top=4cm,bottom=4cm]{geometry}
\usepackage[dvips]{graphicx}
\usepackage[english]{babel}

 \theoremstyle{plain}
\newtheorem{lem}{Lemma}[section]
\newtheorem{thm}[lem]{Theorem}

\newtheorem{cor}[lem]{Corollary}
\theoremstyle{definition}
\newtheorem{Def}[lem]{Definition}
\newtheorem{Rem}[lem]{Remark}

\newcommand{\mck}{\mathcal{K}}
\newcommand{\rp}{\mathbb{R}_{+}}
\newcommand{\Ll}{L_{loc}}
\newcommand{\Llw}{L_{loc,w}}

\newcommand{\bu}{\mathbf{u}}
\newcommand{\bb}{\mathbf{B}}
\newcommand{\bo}{\mathcal{B}_0}
\newcommand{\mb}{\mathcal{B}}
\newcommand{\ve}{\mathbb{V}}
\newcommand{\by}{\mathbf{y}}
\newcommand{\bym}{\mathbf{y}^m}

\newcommand{\h}{\mathbb{H}}
\newcommand{\byms}{\bym(s)}

\newcommand{\bn}{\mathbf{n}}
\newcommand{\bv}{\mathbf{v}}
\newcommand{\bw}{\mathbf{w}}
\newcommand{\mfa}{\mathcal{F}^a_+}
\newcommand{\mfl}{\mathcal{F}^+_{loc}}
\newcommand{\thl}{\Theta^{loc}_+}
 \DeclareMathOperator{\curl}{curl}
\DeclareMathOperator{\Div}{div} \DeclareMathOperator{\Span}{Span}
\DeclareMathOperator{\dist}{\mathbf{dist}}
\newcommand{\CQ}{\mathcal{Q}}
\title[Trajectory attractor for the MHD of Non-Newtonian Fluids]{Trajectory attractor for a non-autonomous Magnetohydrodynamic equation of
Non-Newtonian Fluids}
\author{Paul Andr\'e Razafimandimby}
\email{paulrazafi@gmail.com}
\address{Department of Mathematics and Information-Technology\\
Montan University of Leoben\\ Franz Josefstra\ss e 18, Leoben,
8700, Austria}

\begin{document}
\begin{abstract}
In this article we initiate the mathematical study of the dynamics
of a system of nonlinear partial differential equations modelling
the motion of incompressible, isothermal and conducting modified
bipolar fluids in presence of magnetic field. We mainly prove the
existence of weak solutions to the model. We also prove the
existence of a trajectory attractor to the translation semigroup
acting on the trajectories of the set of weak solutions and that
of external forces. Some results concerning the structure of this
trajectory attractor are also given. The results from this paper
may be useful in the investigation of some system of PDEs arising
from the coupling of incompressible fluids of $p$-structure and
the Maxwell equations.
\end{abstract}
 \maketitle
\section{Introduction}
 For homogeneous incompressible fluids, the
constitutive law satisfies
\begin{equation*}
\mathbb{T}=-\pi\mathbf{1}+\hat{\mathbb{T}}(\mathcal{E}(\bu)),
\end{equation*}
where $\mathbb{T}$ is the Cauchy stress tensor, $\bu$ is the
velocity of the fluid, and $\pi$ is the undetermined pressure due
to the incompressibility
condition,  $\mathbf{1}$ is the identity tensor. The argument tensor $%
\mathcal{E}(\bu)$ of the tensor symmetric-valued function
$\hat{\mathbb{T}}$ is defined  through
\begin{equation*}
\mathcal{E}(\bu)=\frac{1}{2}\left(\mathbf{\ L}+\mathbf{L}^{\text{T}%
}\right),\quad \mathbf{\ L}=\nabla{ \bu},
\end{equation*}
where the T superscript denotes the matrix transpose.
Magnetohydrodynamics (MHD) is a branch of continuum mechanics
which studies the motion of conducting fluids in the presence of
magnetic fields. The system of Partial Differential Equations in
MHD are basically obtained through the coupling of the dynamical
equations of the fluids with the Maxwell's equations which is used
to take into account the effect of the Lorentz force due to the
magnetic field (see for example \cite{Chandrasekhar}). They play a
fundamental role in Astrophysics, Geophysics, Plasma Physics, and
in many other areas in applied sciences. In many of these, the MHD
flow exhibits a turbulent behavior which is amongst the very
challenging problems in nonlinear science. Because of these facts,
MHD has been the object of intensive scientific investigation.
Due to the folklore fact that the Navier-Stokes is
 an accurate model for the motion of incompressible in many practical situation, most of scientists have assumed that the fluids are Newtonian whose reduced stress tensor
$\hat{\mathbb{T}}(\mathcal{E}(\bu))$ is a linear function of $%
\mathcal{E}(\bu)$.

%

 However, there are many conducting materials
 that cannot be characterized as Newtonian fluids. To
describe these media one generally has to use (conducting) fluid models that allow $%
\hat{\mathbb{T}}$ to be a nonlinear function  of $\mathcal{E}(\bu)$. Fluids in
the latter class are called Non-Newtonian fluids. We refer for example to the introduction of
Biskamp's book \cite{BISKAMP}
 for some examples of these Non-Newtonian
 conducting fluids. In \cite{LADY1} and \cite%
{LADY2}, Ladyzhenskaya considered a mathematical model for nonlinear fluids whose reduced
stress tensor $\hat{\mathbb{T}}(\mathcal{E}(\bu))$ satisfies
\begin{equation}\label{fluid-math}
\hat{\mathbb{T}}(\mathcal{E}(\bu))=2(\varepsilon+\mu_0|\nabla \bu|^{p-2})\frac{%
\partial \bu_i}{\partial x_j}, \,\, p\ge 2.
\end{equation}
Since then, this model has been the object of intensive
mathematical analysis which has generated several important
results. We refer to \cite{DU}, \cite{LIONS} for some relevant
examples. In \cite{DU} the authors emphasized important reasons
for considering such model. But the fluids or MHD models finding
source from \eqref{fluid-math} do not really have a meaning in
physics as they do not satisfy some principles of continuum
mechanics and thermodynamics. Necas, Novotny and Silhavy
\cite{NECAS1}, Bellout, Bloom and Necas \cite{BELLOUT2} have
developed the theory of multipolar viscous fluids which was based
on the work of Necas and Silhavy \cite{NECAS2}. Their theory is
compatible with the basic principles of thermodynamics such as the
Clausius-Duhem inequality and the principle of frame indifference,
and their results to date indicate that  the theory of
multipolar fluids may lead to a better understanding of
hydrodynamic
turbulence (see for example \cite{BELLOUT3}). Bipolar fluids whose reduced stress tensor $\hat{%
\mathbb{T}}(\mathcal{E}(\bu))$
 is defined by
\begin{equation}\label{tens-bip}
\hat{%
\mathbb{T}}(\mathcal{E}(\bu))=2\kappa_0(\varepsilon+ |\mathcal{E}(\bu)|^2)^{\frac{p-2}{2}}%
\mathcal{E}(\bu)-2\kappa_1\Delta\mathcal{E}(\bu),
\end{equation}
form a particular class of multipolar fluids.  If $1<p\le 2$ then
the fluids are said to be shear thinning, and shear thickening when
$2<p$.

Throughout our analysis, we suppose the existence of a tensor
valued function $\mathbf{T}:\mathbb{R}^{n\times
n}_{sym}\rightarrow \mathbb{R}^{n\times n}_{sym}$,  a scalar
function $\Sigma: \mathbb{R}^{n\times n}_{sym}\rightarrow
\mathbb{R}^+_0$, and constants $\nu_1,\nu_2$ such that for some
$p>1$ and for all $l, k, i, j=1,2,\ldots,n$, $\mathbf{D},
\mathbf{E}\in \mathbb{R}^{n\times n}_{sym}$:
\begin{align}
 \Sigma(0)=0,\quad \frac{\partial\Sigma(0)}{\partial D_{kl}}=\partial_{kl}\Sigma(0)=0,\quad \mathbb{T}_{kl}(\mathbf{D})=\partial_{kl}\Sigma(\mathbf{D}),
 \label{tens1}\\
\partial_{ij}\partial_{kl}\Sigma(\mathbf{D})E_{ij}E_{kl}\ge \nu_1 (1+\mathbf{D})^{p-2}|\mathbf{E}|^2, \label{tens2}\\
\partial_{ij}\partial_{kl}\Sigma(\mathbf{D})\le \nu_2 (1+|\mathbf{D}|^2)^{p-2}. \label{tens3}
\end{align}
Here
$$\mathbb{R}^{n\times n}_{sym}=\{\mathbf{D}\in \mathbb{R}^{n\times n}: D_{ij}=D_{ji}, i,j=1,2,\dots, n\}.$$

In the present work, we will consider a model of MHD arising from
the coupling of modified bipolar fluids and the Maxwell equations.
More precisely, we assume that the reduced stress tensor of the
fluids is given by
\begin{equation}\label{fluid-phys}
\hat{\mathbb{T}}(\mathcal{E}(\bu))=2\kappa_0\mathbf{T}(\mathcal{E}(\bu))-2\kappa_1\Delta\mathcal{E}(\bu).
\end{equation}
Furthermore, let $\mathcal{Q}$ be a simply-connected, and bounded
domain of $\mathbb{R}^n$ ($n=2,3$) such that the boundary $\partial
Q$ is of class $C^\infty$. This will ensure the existence of a
normal vector $\bn$ at each of its point. In this article we
are aiming to give some mathematical results related to the
following non-autonomous partial differential equations:
\begin{equation}  \label{*}
\begin{cases}
\frac{\partial \bu}{\partial t}
-\nabla\cdot\hat{\mathbb{T}}(\mathcal{E}(\bu))+\bu\cdot \nabla
\bu+\mu \bb\times\curl \bb+ \nabla P =g(x,t)
 \text{ in } (0,T]\times \mathcal{Q},\\
\frac{\partial \bb}{\partial t}+S \curl \curl \bb+\mu \bu\cdot \nabla \bb-\mu \bb\cdot \nabla \bu=0 \text{ in } (0,T]\times \mathcal{Q}\\
\Div \bu=\Div \bb=0,\text{ in } (0,T]\times \mathcal{Q}\\
\bu=\tau_{ijl}\bn_j\bn_l=0 \text{ on }  (0,T)\times\partial\mathcal{Q},\\
\bb\cdot \bn=\curl \bb \times \bn =0 \text{ on } (0,T)\times  \partial\mathcal{Q},\\
\bu(0)=\bu_0, \,\,\, \bb(0)=\bb_0 \text{ in } \mathcal{Q},
\end{cases}%
\end{equation}
where $\bu=(u_i; i=1,\ldots, n)$, $\bb=(B_i;i=1,\dots, n)$ and $P$
are unknown functions defined on $\mathcal{Q}\times [0,T]$,
representing, respectively, the fluid velocity, the magnetic field
and the pressure, at each point of $\mathcal{Q}\times [0,T]$. The
vector $\bn$ represent the normal to $\partial \mathcal{Q}$. The
constants $S$ and $\mu$ are positive constants depending on the
Reynolds numbers of the fluids and the Hartman number. Finally,
$\tau_{ijl}$ is defined by
\begin{equation*}
\tau_{ijl}=\frac{\partial \mathcal{E}_{ij}(\bu)}{\partial x_l}.
\end{equation*}

 The system \eqref{*} is a modified MHD equation that we obtained from the coupling of
the Maxwell equations and the non-autonomous dynamical equations
of a modified isothermal incompressible bipolar fluids. The
structure of the nonlinearity of problem \eqref{*} makes it as
interesting as any nonlinear evolution equations of mathematical
physics such as the basic MHD or the Navier-Stokes equations. In
addition to well-known tools from Navier-Stokes of MHD, new tools
need to be elaborated to handle \eqref{*}.

 When $\bb\equiv0$ then
\eqref{*} reduces to the PDEs describing the motion of isothermal
incompressible nonlinear bipolar fluids which has been thoroughly
investigated during the last two decades. Several important
results were obtained by prominent mathematical researchers (see,
among others, \cite{BELLOUT1}, \cite{BELLOUT4}, \cite{BELLOUT5},
\cite{MALEK}, \cite{MALEKetal}) in this direction of research.
For $p=2,\,\, \kappa_1=0,$ \eqref{*} reduces to MHD equations
which has been the object of intensive mathematical research
since the pioneering work of Ladyzhenskaya and Solonnikov
\cite{LADY-SOLO}. We only cite \cite{STUPELIS},
\cite{TEMAM+SERMANGE}, \cite{LEBRIS}, \cite{DESJARDINS} for few
relevant examples; the reader can consult \cite{GERBEAU+LEBRIS}
for a recent and detailed review.  Assuming that $\kappa_1=0$
Samokhin studied the MHD equations arising  from the coupling of
the Ladyzhenskaya model with the Maxwell equations in
\cite{SAMOKHIN}, \cite{SAMOKHIN2}, \cite{SAMOKHIN3}, and
\cite{SAMOKHIN4}. In these papers he  proved the existence of weak
solution of the model for $p\ge 1+\frac{2n}{n+2}$. Later on
Gunzburger and his collaborators generalized the settings of
Samokhin by taking a fluid with a stress tensor having a more
general $p$-structure in \cite{GUNZBURGER} and \cite{GUNZBURGER2}.
The authors of the later papers analyzed the well-posedness and
the control of \eqref{*} still in the case where $\kappa_1=0$ and
$p\ge 1+\frac{2n}{n+2}$.

In this work we assume that $\{\kappa_1\neq 0, \bb\not \equiv 0,
p\in (1,1+\frac{2n}{n+2}]\}$ we are interested in the analysis of
the long-time dynamics of the weak solutions to \eqref{*} which is
very important for the understanding of the global temporal
behavior and the physical features (such as the turbulence in
hydrodynamics) of the model. We refer, for instance, to
\cite{BABIN+VISHIK}, \cite{ROBINSON}, and \cite{TEMAM-INF} for
some results in this direction for the case of autonomous
Navier-Stokes and other autonomous equations of mathematical
physics.

In the third section of this paper, we first give an existence result
for the weak solutions to \eqref{*} in both two and three
dimensional space and $p \in (1,\frac{2n+6}{n+2}]$. More
precisely, by means of a blending of the Galerkin and
monotonicity-compactness methods we could prove the following
\begin{thm}
Let $\mathbf{T}$ and $\Sigma$ satisfy \eqref{tens1}-\eqref{tens3},
$(\bu_0;\bb_0)\in \h$, $g\in L^2(0,T;\ve^\ast)$ and  $p\in
(1,\frac{2n+6}{n+2}]$. Then there exists at least one pair $(\bu;
\bb)\in L^\infty(0,T;\h)\cap L^2(0,T;\ve)$ such that
\begin{equation*}
 \left(\frac{\partial \bu}{\partial t}; \frac{\partial \bb}{\partial t}\right)\in L^2(0,T;\ve^\ast),
\end{equation*}
and \eqref{*} holds in the weak sense.
\end{thm}
The notations used in the above statement will be explained clearly in the next two sections.

The long-time behavior of weak solutions to non-autonomous
evolutionary PDEs \eqref{*} is analyzed in the last section. 
The classical and well-known approach
explained in the monographs \cite{BABIN+VISHIK}, \cite{ROBINSON},
and \cite{TEMAM-INF} does not apply here since we are considering
a non-autonomous system. Moreover, when $p\in
(1,\frac{2n+6}{n+2}]$ then we do not have a uniqueness result.
Therefore, the classical theory of forward attractor developed
in \cite{VISHIK+CHEPYZHOV-94}, \cite{Haraux}, \cite{VISHIK} and the concept of pullback attractor initially elaborated in
\cite{Crauel, Craueletal, Schmalfuss} for non-autonomous
dissipative dynamical systems does not apply here. To overcome the difficulties associated with 
the possible non-uniqueness of solution in the analysis of dynamical system, on can usually use three methods. The first one is the method of 
generalized flow which was initially developed in \cite{Ball-1978}. 
The second approach is the theory of multivalued semi-flows which has already found in the literature in 1948
(see, for instance, \cite{Barbashin-1948}). However, this method was applied to the analysis of the asymptotic behavior of PDEs for the first time
 in 
\cite{Babin+Vishik-1985}
 and extended later on in other work (\cite{Babin-1995}, \cite{Melnik+Valero-1998}). One can also cite the paper \cite{Caraballo+al-2003b} which extended the notion of pullback attractor
 to nonautonomous and stochastic multivalued 
 dynamical systems. The third method is the theory of trajectory attractor which was introduced in \cite{VISHIK+CHEPYZHOV} , 
 \cite{Malek+Necas-1996}, \cite{Sell-1996}. Due to its simplicity this method has
become very popular and used in many works but we only mention
\cite{VISHIK+CHEPYZHOV}, \cite{VISHIK+CHEPYZHOV-2002},
\cite{VISHIK+CHEPYZHOV+WENDLAND}, \cite{Kapustyan}, and
\cite{ZHOU} for few relevant results. We can find a comparison of the two first  and the three methods in \cite{Caraballo+al-2003a} and 
 \cite{Kapustyan-2010}, respectively. One can also find a recent review on the three methods in \cite{Caraballo+al-2010}. 
 In the present work we use the third approach and we show the
existence of the trajectory attractor of weak solutions to
\eqref{*} and give some partial results related to its structure.
More precisely,  we obtain the following which will be presented
in detail in  Section 4.
\begin{thm}
Let $g_0$ be translation bounded in $\Ll^2(\rp; \ve^\ast)$ (that
is, it is translation compact in $\Llw^2(\rp; \ve^\ast)$), then
the translation semigroup $\{\mathbb{S}_t, t\ge 0\}$
 acting on $\mck_{\Sigma(g_0)}^+$ possesses a uniform trajectory attractor $\mathcal{A}_{\Sigma(g_0)}$. The set $\mathcal{A}_{\Sigma(g_0)}$
 is bounded in $\mathcal{F}^a_+$ and compact in $\Theta^{loc}_+$. It is strictly invariant:
\begin{equation}
 \mathbb{S}_t\mathcal{A}_{\Sigma(g_0)}=\mathcal{A}_{\Sigma(g_0)},
\end{equation}
for all $t\ge 0$. Furthermore, we have
 \begin{equation}
\mathcal{A}_{\Sigma(g_0)}=\mathcal{A}_{\omega(\Sigma(g_0))},
 \end{equation}
 where $\mathcal{A}_{\omega(\Sigma(g_0))}$ is the uniform (with respect to $g\in \omega(\Sigma(g_0))$)
  attractor of $\{\mck_g^+, g\in \omega(\Sigma(g_0) \}$.

Moreover, for any set $\mathbb{B}\subset\mck^+_{\Sigma(g_0)}$ which is bounded in $\mfa$ we have
\begin{align}
 \dist_{C(0,T; D(\mathcal{A}^{-\delta}))}(\Pi_{0,T}\mathbb{S}_T\mathbb{B}, \Pi_{0,T}\mathbb{K}_{Z(g_0)})\rightarrow 0 \text{ as } t\rightarrow \infty,\\
 \dist_{L^2(0,T; D(\mathcal{A}^{\delta}))}(\Pi_{0,T}\mathbb{S}_T\mathbb{B}, \Pi_{0,T}\mathbb{K}_{Z(g_0)})\rightarrow 0 \text{ as } t\rightarrow \infty.
\end{align}
\end{thm}
We should mention that even if we drew our inspiration from
\cite{VISHIK+CHEPYZHOV}, \cite{DUVAUT+LIONS}, and
\cite{TEMAM+SERMANGE} the problem we treated here does not fall in
the framework of these main references.
 Besides the usual nonlinear terms of the MHD equations
it contains another nonlinear term of $p$-structure  which
exhibits the non-linear relationships between the reduced stress
and the rate of strain $\mathcal{E}(\bu)$ of the conducting
fluids. Because of this, the analysis of the behavior of the MHD
model \eqref{*} tends to be much more complicated and subtle than
that of the Newtonian MHD equations. Hence, we have had to invest
much effort  to prove many important results which do not follow
from the analysis in the latter works or we could not find in the
literature. All of our findings rely on these crucial results and
they are presented in Section 2.

 To the best of our knowledge the present article is the first to deal with \eqref{*}
 in the setting $\{\kappa_1\neq 0, \bb\not \equiv 0, p\in (1,1+\frac{2n}{n+2}]\}$.
In this sense, many topics and problems are still open.
Some examples of challenges we may address in future research are
the existence of weak solution for all values of $p$, the
uniqueness of such weak solutions. We may also want to study the
behavior of the weak solution we obtained in this paper as
$\kappa_1$ approaches to 0 but fixing $p \in (1,2)$; this is a
very interesting research topic and may lead to an extension of
the results obtained in \cite{SAMOKHIN}, \cite{SAMOKHIN2},
\cite{SAMOKHIN3}, \cite{SAMOKHIN4}, \cite{GUNZBURGER} and
\cite{GUNZBURGER2}. These few examples of research topics are
taken as an analogy of the problems still unsolved in the
mathematical theory of generalized Non-Newtonian fluids as reported in
\cite{BELLOUT4}, \cite{RUZICKA} and \cite{MALEKetal}. All of these
questions are very difficult and beyond the scope of this paper,
thus we will just limit ourselves with giving a suitable
mathematical setting for \eqref{*} and partial results related to
the dynamics of the weak solutions. However, we hope that our work
will find its applications elsewhere.

%

The organization of this article is as follows. We introduce the necessary notations for the mathematical theory of
\eqref{*} in the next section; these notations enabled us to rewrite \eqref{*} into an abstract evolution equations.
In the very same section we also give very important results that we could not find in the literature. Our principal claims rely very much on these new tools. Section 3 is devoted to the proof of the existence of weak solutions to
 \eqref{*} whose long-time behavior will be investigated in the last section of the article.

\noindent \newblock{\em Hereafter we should make the convention
that the problem \eqref{*} is characterized by its data:
\begin{enumerate}
 \item physical data: $\kappa_0, \kappa_1,\mu, \nu_1,\nu_2, S,  p$, $n$,
 $\mathcal{Q}$,
\item initial conditions $\bu_0$ and $\bb_0$,
\item and the
external force $g$.
\end{enumerate}
Unless the contrary is mentioned, any positive constant ($C$,
$\widetilde{C}_i$, $C_i$, $\lambda$, etc) depend at most on the
data of \eqref{*}. Moreover, they may change form one term to the
other.}
\section{Preliminary: Notations and hypotheses}\label{Prel}

We introduce some notations and background following the
mathematical theory
of hydrodynamics (see for instance \cite{MALEKetal}). For any $p\in [1,\infty)$, $\mathbb{L}^p(\mathcal{Q})$ and $\mathbb{W}^{m,p}(\mathcal{Q}%
)$ are the spaces of functions taking values in $\mathbb{R}^n$
such that each component belongs to the Lebesgue space
$L^p(\mathcal{Q})$ and the Sobolev spaces $W^{m,p}(\mathcal{Q})$,
respectively. For $p=2$ we use $\h^m(\mathcal{Q})$ to describe
 $\mathbb{W}^{m,p}(\mathcal{Q})$.
The symbols $|\cdot|$ and $(.,.)$ are the $\mathbb{L}^2$-norm and $\mathbb{L}^2$-inner product, respectively. The norm of $\mathbb{W}^{m,p}(\mathcal{Q}%
)$ is denoted by $||\cdot||_{m,p}$. As usual,
$\mathbb{C}^\infty_0(\mathcal{Q})$ is the space of infinitely
differentiable functions having compact support contained in
$\mathcal{Q}$. The space $\mathbb{W}^{p,m}_0(\mathcal{Q})$ is the closure of
$\mathbb{C}^\infty_0(\mathcal{Q})$ in $\mathbb{W}^{p,m}(\mathcal{Q})$. Now we
introduce the following spaces
\begin{align*}
 \mathcal{V}_1 &=\left\{\bu\in \mathbb{C}^\infty_0(\mathcal{Q}): \Div \bu=0\right\},\\
\mathbb{H}_1 &=\left\{\bu \in \mathbb{L}^2(\mathcal{Q}): \Div \bu =0, \bu\cdot \bn= 0 \text{ on } \partial \mathcal{Q} \right\},\\
\mathbb{V}_1&=\left\{ \bu \in \mathbb{H}^2(\mathcal{Q}): \Div
\bu=0, \bu=\frac{\partial \bu }{\partial \bn}=0 \text{ on }
\partial \mathcal{Q}\right\}.
\end{align*}
We also set
\begin{align*}
 \mathcal{V}_2&=\left\{\bb \in \mathbb{C}^\infty_0(\mathcal{Q}): \Div \bb=0; \bb\cdot \bn=0 \text{ on } \partial \mathcal{Q}\right\},\\
\mathbb{H}_2&=\text{ the closure of $\mathcal{V}_2$ in $\mathbb{L}^2(\mathcal{Q})$},\\
\mathbb{V}_{2}&=\left\{ \bb\in \mathbb{H}^1(\mathcal{Q}): \Div
\bb=0; \bb\cdot \bn=0 \text{ on } \partial \mathcal{Q}\right\}.
\end{align*}
Note that
\begin{equation*}
 \mathbb{H}_1=\mathbb{H}_2.
\end{equation*}

The spaces $\mathbb{H}_i, i=1,2$ are equipped with the scalar
product and norm induced by $\mathbb{L}^2(\mathcal{Q})$.


We endow the space $\mathbb{V}_{1}$ with the norm $||\cdot||_{1}$
generated by the scalar product
\begin{equation*}
 ((\bu, \mathbf{v}))_1=\int_\mathcal{Q} \frac{\partial \mathcal{E}_{ij}(\bu)}{%
\partial x_k}\frac{\partial \mathcal{E}_{ij}(\bv)}{\partial x_k} dx,
\end{equation*}
where \begin{equation*}
\mathcal{E}(\bu)=\frac{1}{2}\left(\mathbf{\ L}+\mathbf{L}^{\text{T}%
}\right),\quad \mathbf{\ L}=\nabla{ \bu}.
\end{equation*}
It is shown in \cite{BELLOUT5} that this scalar product generates a norm $||\cdot||_1$ which is equivalent to the usual $\h^2(\mathcal{Q})$-norm on $\ve_1$.
More precisely, there exists two positive constants $K_1$ and $K_2$ depending only on $\CQ$ such that
\begin{equation}\label{eq-bloom}
 K_1 ||\bu||_{2,2}\le ||\bu||_1\le K_2 ||\bu||_{2,2},
\end{equation}
for any $\bu\in \ve_1$.

On $\mathbb{V}_{2}$ we define the scalar product
\begin{equation*}
 ((\bu, \mathbf{v}))_2=(\curl \bu, \curl \mathbf{v}),
\end{equation*}
which coincides with the usual scalar product of
$\mathbb{H}^1(\mathcal{Q})$.

Let
\begin{align*}
 \mathbb{V}=\mathbb{V}_{1}\times \mathbb{V}_2,\\
\mathbb{H}=\mathbb{H}_1\times \mathbb{H}_2.
\end{align*}
The spaces $\mathbb{H}$ and $\ve$ have the structure of Hilbert
spaces when equipped respectively with the scalar products
\begin{equation*}
 (\Phi,\Psi)=(\bu, \mathbf{v})+(\bb, \mathbf{C}), \forall \Phi=(\bu; \bb), \Psi=(\mathbf{v}; \mathbf{C})\in
\mathbb{H},
\end{equation*}
 and
 \begin{equation*}
((\Phi,\Psi))=2\kappa_1 ((\bu,\bv))_1+S((\bb,\mathbf{C}))_2, \forall
\Phi=(\bu; \bb), \Psi=(\mathbf{v}; \mathbf{C})\in \mathbb{V}.
 \end{equation*}

The norms on $\h$ and $\mathbb{V}$ are respectively defined by
\begin{equation*}
|\Phi|=\sqrt{|\bu|^2+|\bb|^2}, \forall \Phi=(\bu;\bb)\in \h,
\end{equation*}
and
\begin{equation*}
 ||\Phi||=\sqrt{2\kappa_1||\bu||^2_{1}+S||\bb||^2_2},\forall \Phi=(\bu; \bb)\in \ve.
\end{equation*}

\begin{Rem}\label{Rem1}
Note that the norm $||\Phi||$ is equivalent to the norm defined by
\begin{equation*}
 [[\Phi]]^2=||\bu||^2_{1}+||\bb||^2_2.
\end{equation*}
We should also mention that a Poincar\'e'-like inequality holds for the space $\h$ and $\ve$, that is there exists a positive constant $\lambda$ ( dependong only on $\mathcal{Q}$)
 such that
\begin{equation}\label{Poincare}
 \lambda |\bu|^2\le ||\bu||^2,
\end{equation}
for any $\bu\in \ve$.
\end{Rem}

For any Banach space $X$ we denote by $X^\ast$ its dual space and
$\langle \phi, \bu\rangle$ the value of $\phi\in X^\ast$ on $\bu
\in X$. Identifying $\mathbb{H}$ with its dual, we have the following
Gelfand chain
\begin{equation*}
 \mathbb{V}\subset \mathbb{H}\equiv \h^\ast \subset \mathbb{V}^\ast,
\end{equation*}
 where each space is densely and compactly embedded into the next one. This chain of embedding enables us to write
\begin{equation*}
 \langle \bu,\bv\rangle=(\bu,\bv),
\end{equation*}
for any $\bu\in \h$ and $\bv\in \ve$.

Let $\mathcal{A}_1$ be a linear operator defined through the
relation
\begin{equation*}
\langle \mathcal{A}_1\bu,\bv\rangle=a(\bu,\bv)=2\kappa_1\int_D \frac{\partial \mathcal{E}_{ij}(\bu)}{%
\partial x_k}\frac{\partial \mathcal{E}_{ij}(\bv)}{\partial x_k} dx, \,\, \bu\in
D(\mathcal{A}_1), \bv\in \ve_1.
\end{equation*}
Note that
\begin{equation*}
D(\mathcal{A}_1)=\lbrace \bu\in \ve: \exists \mathbf{f}\in
\h\subset \ve^\ast \text{ for which } a(\bu,\bv)=(\mathbf{f},\bv),
\forall \bv\in \ve\rbrace,
\end{equation*}
 and $\mathcal{A}_1=\mathbf{P}\Delta^2$, where $\mathbf{P}$ is the
orthogonal projection defined on $L^2(D)$ onto $\h$. The operator $\mathcal{A}_1$ is self-adjoint and it follows from Rellich's theorem that it is compact on $\h_1$.
Therefore, there exists
a sequence of positive numbers $\{\lambda_i: i=1, 2, 3,\dots\}$ and a family of smooth function $\{\phi_i: i=1, 2, 3, \dots\}$ satisfying
\begin{equation}\label{spec-a-1}
 \mathcal{A}_1\phi_i=\lambda_i \phi_i,
\end{equation}
for any $i$. We can normalize the family $\{\phi_i: i=1, 2, 3,
\dots\}$ so that they will form an orthonormal basis of $\h_1$
which is orthogonal in $\ve_1$.

 We also
introduce a linear operator $\mathcal{A}_2$ from $\mathbb{V}_2$
taking values into $\mathbb{V}_2^\ast$ (i.e, $\mathcal{A}_2\in
\mathcal{L}(\mathbb{V}_2,\mathbb{V}_2^\ast)$ ) defined by
\begin{equation*}
 \langle \mathcal{A}_2 \bb, \mathbf{C}\rangle= S ((\bb, \mathbf{C}))_2,
\end{equation*}
 for any $\bb, \mathbf{C}\in \mathbb{V}_2$. The operator $\mathcal{A}_2$ is self-adjoint and compact on $\h_2$, so as before we can find a family of increasing positive numbers
$\{\mu_i: i=1, 2, 3, \dots\}$
  and a family of smooth functions $\{ \psi_i: i=1, 2, 3, \dots\}$ such that
\begin{equation}\label{spce-a-2}
 \mathcal{A}_2\psi_i=\mu_i \psi_i,
\end{equation}
 holds for any $i$.
It is known from \cite{TEMAM+SERMANGE} that the family $\{ \psi_i:
i=1, 2, 3, \dots\}$ verifies the following problem
\begin{equation*}
 \begin{cases}\curl \curl \psi_i =\mu_i \psi_i,\\
\Div \psi_i=0,\\
(\psi_i\cdot \bn)_{\partial \CQ}= (\curl \psi_i\cdot \bn)_{\partial \CQ}=0,
 \end{cases}
\end{equation*}
for any $i$.

Through $\mathcal{A}_1$ and $\mathcal{A}_2$ we can define a linear
operator from $\ve$ to $\ve^\ast$ by setting
\begin{equation}\label{def-A}
\langle \mathcal{A}\Phi,\Psi\rangle=\langle
\mathcal{A}_1\bu,\bv\rangle+\langle\mathcal{A}_2\bb,\mathbf{C}\rangle,
\end{equation}
for any $\Phi=(\bu; \bb), \Psi=(\mathbf{v}; \mathbf{C})\in
\mathbb{V}$.

%

To take into account the $p$-structure of the fluids we introduce
a nonlinear mapping $\mathcal{A}_p$ from $\mathbb{V}$ into
$\mathbb{V}^\ast$ by setting
\begin{equation*}
 \langle \mathcal{A}_p \Phi, \Psi\rangle=\int_\mathcal{Q} \mathbf{T}(\mathcal{E}(\bu)) \cdot \mathcal{E}(\mathbf{v}) dx,
\end{equation*}
 for any $\Psi=(\bu;\bb), \Psi=(\mathbf{v};\mathbf{C}) \in \mathbb{V}$.
  We state some important properties of $\mathcal{A}_p$
in the following
\begin{lem}\label{lem1}
Let $p\in (1,\frac{2n+6}{n+2}]$ and let $\mathbf{T}$ and $\Sigma$
satisfy \eqref{tens1}-\eqref{tens3}. Then,
\begin{enumerate}
 \item[(i)] the nonlinear operator $\mathcal{A}_p$ is monotone; that is,
\begin{equation*}
\langle \mathcal{A}_p\Phi_1-\mathcal{A}_p\Phi_2,
\Phi_1-\Phi_2\rangle \ge 0,
\end{equation*}
for any $\Phi_1, \Phi_2 \in \mathbb{V}$. In particular, we have
\begin{equation}\label{4}
 \langle \mathcal{A}_p \Phi, \Phi\rangle \ge 0,
\end{equation}
for any $\Phi \in \mathbb{V}$.
 \item[(ii)] If $\Phi \in
L^2(0,T,\ve)\cap L^\infty(0,T,\h)$, then $\mathcal{A}_p \Phi\in
L^2(0,T,\ve^\ast)$.
\end{enumerate}
\end{lem}
To check the results in the above lemma we need to recall the following results whose proofs can be found in
\cite{RUZICKA}.
\begin{lem}[\textbf{Korn's inequalities}]
 Let $1<p<\infty$ and let $\mathcal{Q}\subset \mathbb{R}^n$ be of class $C^1$. Then there exist two positive constants $K^i_p=K^i_p(\mathcal{Q}), i=1,2$ such that
\begin{equation}
 K^1_p ||\bu||_{1,p}\le \left(\int_\mathcal{Q} |\mathcal{E}(\bu)|^p dx\right)^\frac{1}{p}\le K^2_p ||\bu||_{1,p},
\end{equation}
for any $\bu\in \ve_{1,p}$.
\end{lem}
 \begin{proof}[Proof of Lemma \ref{lem1}]
It is known from \cite{MALEKetal} that for any $p\in(1,\infty)$
there exists a positive constant $\nu_3$ depending only on the
physical data such that for all $\mathbf{D}, \mathbf{E}\in
\mathbb{R}^{n\times n}_{sym}$:
\begin{align}
\left(\mathbf{T}(\mathbf{D})-\mathbf{T}(\mathbf{E})\right)\cdot
\left(\mathbf{D}-\mathbf{E}\right)\ge 0 \label{l4},
\end{align}
and
\begin{equation}
 |\mathbf{T}(\mathbf{D})|^2\le \nu_3 (1+|\mathbf{D}|)^{2p-2}.\label{eq13}
\end{equation}

 Therefore, we see from \eqref{l4} that for any $\bu,\mathbf{v}\in \ve_{1}$
\begin{align}
 \langle\mathcal{A}_p \bu-\mathcal{A}_p \mathbf{v}, \bu-\mathbf{v}\rangle=&\int_\mathcal{Q}\left[\mathbf{T}(\mathcal{E}(\bu))-\mathbf{T}(\mathcal{E}(\mathbf{v}))\right]\cdot \left[\mathcal{E}(\bu)-\mathcal{E}
(\mathbf{v})\right]dx\ge 0,\nonumber
\end{align}
which proves the monotonicity of $\mathcal{A}_p$.

By noticing that $\mathbf{T}(0)=0$ the estimate \eqref{4} is a
simple consequence of the last inequality.

For any $\Psi=(\bv;\mathbf{C}), \Phi=(\bu;\bb) \in \ve$ we have
\begin{equation*}
 \langle \mathcal{A}_p\Phi,\Psi\rangle=\int_\mathcal{Q}\mathbf{T}(\mathcal{E}(\bu))\cdot \mathcal{E}(\bv) dx,
\end{equation*}
which leads to
\begin{equation*}
 |\langle \mathcal{A}_p\Phi,\Psi\rangle| \le\left(\int_\mathcal{Q}|\mathbf{T}(\mathcal{E}(\bu))|^2dx \right)^\frac{1}{2}
\left(\int_\mathcal{Q}|\mathcal{E}(\bv)|^2 dx \right)^\frac{1}{2}.
\end{equation*}
This last estimate along with the discrete H\"older's inequality
and Korn's inequality imply that there exists  a constant $C>0$
such that
\begin{equation}\label{ap-est1}
 |\langle \mathcal{A}_p\Phi,\Psi\rangle| \le C \left[\int_\mathcal{Q}|\mathbf{T}(\mathcal{E}(\bu))|^2 dx \right]^\frac{1}{2}
\left[||\mathbf{C}||^2_1+ |\nabla \bv|^2\right]^\frac{1}{2}.
\end{equation}
As $\ve_1$ is continuously embedded in $\h^1_0(\mathcal{Q})$, we see from \eqref{ap-est1} that
 \begin{equation*}%
 |\langle \mathcal{A}_p\Phi,\Psi\rangle| \le C \left[\int_\mathcal{Q}|\mathbf{T}(\mathcal{E}(\bu))|^2 dx \right]^\frac{1}{2}
||\Psi||.
\end{equation*}
Hence
\begin{equation}\label{ap-est2}
 ||\mathcal{A}_p\Phi||^2_{\ve^\ast}\le C\int_\mathcal{Q}|\mathbf{T}(\mathcal{E}(\bu))|^2 dx.
\end{equation}
Let
\begin{equation*}
 \chi(\bu)=\int_\mathcal{Q}|\mathbf{T}(\mathcal{E}(\bu))|^2 dx.
\end{equation*}
To deal with $\chi(\bu)$ we will distinguish two cases

\noindent\underline{\textsc{Case 1: $1<p \le 2$}}

\noindent We can deduce from \eqref{eq13} that there exists a
positive constant $C$ such that
\begin{equation}
 |\mathbf{T}(\mathbf{D})|^2\le C
 (1+|\mathbf{D}|^2)(1+|\mathbf{D}|^2)^{p-2},\label{12}
\end{equation}
from which along with $p\in(1,2]$ we deduce that
\begin{equation*}
 \chi(\bu)\le C+C \int_\mathcal{Q} |\mathcal{E}(\bu)|^2 dx.
\end{equation*}
Owing to Korn's inequality and the continuous embedding of $\ve_1$ into $\h^1_0(\mathcal{Q})$ we infer that
\begin{equation*}
 \chi(\bu)\le C ||\bu||^2_1.
\end{equation*}
Thus, we derive from the last estimate and \eqref{ap-est2} the existence of a positive constant $C$  such that
\begin{equation*}
 ||\mathcal{A}_p\Phi||^2_{\ve^\ast}\le C+ C ||\Phi||^2,
\end{equation*}
which implies that $\mathcal{A}_p\Phi\in L^2(0,T;\ve^\ast)$ if
$\Phi\in L^2(0,T;\ve)$. The proof of item (ii) is finished for the
case $1<p\le 2$.

\noindent\underline{\textsc{Case 1: $2<p \le \frac{2n+6}{n+2}$}}

\noindent Throughout this step we set $q=2p-2$. For $\bu\in \ve_1$
we infer from Gagliardo-Nirenberg's inequality that there exists a
constant $C>0$ such that
\begin{equation}\label{ap-est4}
 |\nabla \bu|_{r}=\left(\int_\mathcal{Q} |\nabla \bu|^r dx\right)^\frac{1}{r}\le C ||\bu||_1^a |\bu|^{1-a},
\end{equation}
with $\frac{1}{2}\le a\le 1$ and
\begin{equation*}
 r\in \begin{cases}
   [2,\infty) \text{ if } n=2,\\
[2,s) \text{ if } n=3,
  \end{cases}
\end{equation*}
where
\begin{equation*}
 \frac{1}{s}=\frac{1}{n}+\frac{1}{2}-\frac{2a}{n}.
\end{equation*}
Let $q=r\in [2,\infty)$ such that
$\frac{1}{q}=\frac{1}{n}+\frac{1}{2}-\frac{2a}{n}$ with
$\frac{1}{2}\le a \le 1$. This is equivalent to saying that
$$ qa=\frac{q(2+n)-2n}{4}.$$
The estimate \eqref{ap-est4} implies that
\begin{equation*}
 \int_0^T|\nabla \bu|^q_q dt\le C\sup_{s\in [0,T]}|\bu(s)|^{(1-a)q}\int_0^T ||\bu||_1^{qa} dt.
\end{equation*}
Since $p\in [2,\frac{2n+6}{n+2}]$, then it is not difficult to check that $qa<2$ which enables us to apply H\"older's inequality and infer that
\begin{equation}\label{ap-est5}
 \int_0^T|\nabla \bu|^q_q dt\le  C\sup_{s\in [0,T]}|\bu(s)|^{(1-a)q}\left(\int_0^T ||\bu||_1^2 dt\right)^\frac{qa}{2}.
\end{equation}
By using Korn's inequality and the embedding of $\ve_1$ into
$\h^1_0(\mathcal{Q})$ we see from  \eqref{ap-est2} and
\eqref{eq13} that
\begin{equation}\label{ap-est6}
 \int_0^T \chi(\bu(s))ds\le C + C\int_0^T |\nabla\bu(s)|^q_q ds.
\end{equation}
Therefore by plugging \eqref{ap-est5} into the last estimate we get that
\begin{equation}\label{ap-est7}
 \int_0^T \chi(\bu(s))ds\le C +C\sup_{s\in [0,T]}|\bu(s)|^{(1-a)q}\left(\int_0^T ||\bu||_1^2 dt\right)^\frac{qa}{2},
\end{equation}
Noticing that $$|\bu(s)|^{(1-a)q}\le [|\bu(s)|^2+|\bb(s)|^2]^{\frac{(1-a)q}{2}},$$
and $$ \left(\int_0^T ||\bu||_1^2 dt\right)^\frac{qa}{2}\le C \left(\int_0^T \left[2\kappa_1 ||\bu(s)||_1^2 +S||\bb(s)||_2^2\right]dt\right)^\frac{qa}{2},$$
we see that
\begin{equation}
 \int_0^T \chi(\bu(s))ds\le C +C\sup_{s\in [0,T]}|\Phi(s)|^{(1-a)q}\left(\int_0^T ||\Phi(s)||^2 dt\right)^\frac{qa}{2},
\end{equation}
from which it follows that if $\Phi\in L^\infty(0,T;\h)\cap
L^2(0,T;\ve)$ then $\mathcal{A}_p\in L^2(0,T; \ve)$. Hence the
claim (ii) of the lemma holds true.

%
%
 \end{proof}
Before we proceed further, let us state the following
\begin{Rem}\label{rem-1}
 We could see from the course of the proof of Lemma \ref{lem2} that  there exist positive constants 
 $\tilde{C}_1$ and $\tilde{C}_2$ depending only on $p$, $S$ and $\mathcal{Q}$ such that for any function
$\Phi\in L^\infty(t,t+1;\h)\cap L^2(t,t+1;\ve)$ and $t\ge 0$
\begin{equation}\label{est-rem}
\int_t^{t+1} ||\mathcal{A}_p\Phi||_2 \le
  \tilde{C}_1 \int_t^{t+1}||\Phi(s) ||^2+\tilde{C}_2,\,\,\, \text{ if } p\in (1,2],\\
\end{equation}
and
\begin{equation}\label{est-rem2}
 \int_t^{t+1} ||\mathcal{A}_p\Phi||_2 \le \tilde{C}_1\sup_{\substack{\\s\in [0,T]}}|\Phi(s)|^{(1-a)q}\left(\int_0^T
||\Phi(s)||^2 dt\right)^\frac{qa}{2}+\tilde{C}_2,\text{ if
} p\in (2,\frac{2n+6}{n+2}],
\end{equation}
 where $q=2p-2$ and
 $qa=\frac{q(2+n)-2n}{4}$.
\end{Rem}






As in the case of Navier-Stokes equations we introduce the
well-known trilinear form $b(\bu,\bv,\bw)$ to deal with the other
nonlinear terms of \eqref{*}. For any $\bu, \mathbf{v},
\mathbf{w}\in \mathcal{C}^{\infty}_0(\CQ),$ we set
\begin{equation*}
 b(\bu,\mathbf{v},\mathbf{w})=\int_\mathcal{Q} \bu_i\frac{\partial \bv_j}{\partial x_i} \bw_j
 dx.
\end{equation*}
where summations over repeated indices are enforced. It is also
possible to extend the definition  of the trilinear $b(., ., .)$
to larger spaces by exploiting the density of
$\mathcal{C}^{\infty}_0(\CQ)$( or $\mathcal{V}_1$ and
$\mathcal{V}_2$) in appropriate space. We will do it very often
especially each time the trilinear form $b(.,.,.)$ is continuous.
For instance, It is well-known (see, among others, \cite{Temam})
that the trilinear form $b(\bu,\mathbf{v},\mathbf{w})$ is
continuous on $\mathbb{H}^1(\mathcal{Q})\times
\mathbb{H}^1(\mathcal{Q})\times \mathbb{H}^1(\mathcal{Q})$.
Moreover,
\begin{align}
 b(\bu,\mathbf{v},\mathbf{v})&=0,\label{7}\\
b(\bu,\mathbf{v},\mathbf{w})&=-b(\bu,\mathbf{w},\mathbf{v}),
\label{8}
\end{align}
for any $\bu \in \mathbb{V}_2$, $\mathbf{v},\, \mathbf{w}\in
\mathbb{H}^1(\mathcal{Q})$. Since $\mathbb{V}_1\subset
\mathbb{V}_2$, then \eqref{7} and \eqref{8} are also valid for any
element $\bu$ in $\mathbb{V}_1$.

It is also easy to check that
\begin{align}
\left|b(\bu,\bv,\bw)\right|\le&
                   C |\bu||\nabla
\bv| |\bw|_{\mathbb{L}^\infty(\CQ)},\label{ext-trib-A}
\end{align}
\text{ for } $\bu\in \mathbb{L}^2(\CQ), \bv\in \h^1_0(\CQ), \bw
\in \mathcal{C}^{\infty}_0(\CQ),$ and
\begin{align}
\left|b(\bu,\bv,\bw)\right|\le C&  |\bu|_{\mathbb{L}^\infty(\CQ)}
|\bv| |\nabla \bw|, \label{ext-trib-B}
\end{align}
 \text{ for } $\bu\in \mathcal{C}^{\infty}_0(\CQ), \bv \in
\mathbb{L}^2(\CQ), \bw\in \h^1_0(\CQ)$.

Now we introduce a trilinear form defined on $\ve\times\ve\times \ve$ by setting
\begin{equation}\label{9}
 \bo(\Phi_1, \Phi_2, \Phi_3)=b(\bu_1, \bu_2, \bu_3)-\mu b(\bb_1, \bb_2, \bu_3)+\mu b(\bu_1, \bb_2, \bb_3)-\mu b(\bb_1, \bu_2, \bb_3),
\end{equation}
for any $\Phi_i=(\bu_i; \bb_i)\in \ve, i=1,2,3.$

\noindent We collect some properties of $\bo$ in the following lemma.
\begin{lem}\label{lem2}
 \begin{enumerate}
  \item[(i)] For any $\Phi_1, \Phi_2\in \ve$, there exists a bilinear form $\mb(\Phi_1, \Phi_2)$ taking values in $\ve^\ast$ such that
\begin{equation}\label{ex-bilB}
 \bo(\Phi_1, \Phi_2, \Phi_3)=\langle \mb(\Phi_1, \Phi_2), \Phi_3\rangle, \text{ for } \Phi_3 \in \ve.
\end{equation}
\item[(ii)] For any $\Phi_i \in \ve, i=1,2,3$ we have
\begin{align}
 \bo(\Phi_1,\Phi_2, \Phi_3)&=0,\label{10}\\
\bo(\Phi_1, \Phi_2, \Phi_3)&=-\bo(\Phi_1, \Phi_3.
\Phi_2)\label{11}
\end{align}
Furthermore, there exists a constant $C$ depending only on the data of \eqref{*} such that
\begin{equation}\label{ext-trib6}
 \begin{split}
  |\langle\mb(\Phi_1,\Phi_2), \Phi_3\rangle|\le C\max(1,\mu) ||\phi_3||\Big[2 |\Phi_1|\,\, ||\Phi_2||+ ||\Phi_1|||\Phi_2|\Big],
 \end{split}
\end{equation}
for any $\Phi_i \in \ve, i=1,2,3$.
\item[(iii)] If
$\Phi=(\bu;\bb)\in L^2(0,T; \ve)\cap L^\infty(0,T;\h)$ then
$\mb(\Phi,\Phi) \in L^2(0,T; \ve^\ast)$.
 \end{enumerate}
\end{lem}
\begin{proof}
We split the proof into three parts.

\noindent \underline{\textit{Proof of item (i)}:}

\noindent  It is well known that there exists $C>0$ such that
\begin{equation*}
 |b(\bu,\bv,\bw)|\le C ||\bu||_{1,2} ||\bw||_{1,2}||\bw||_{1,2},
\end{equation*}
holds for any $\bu, \bv, \bw\in \h^1(\CQ)$.  Hence by using a
discrete version of H\"older's inequality we have 
\begin{equation}\label{bo-cont}
 |\bo(\Phi_1, \Phi_2, \Phi_3)|\le C\max(1,\mu)||\Phi_1||_{1,2}||\Phi_2||_{1,2}||\Phi_3||_{1,2},
\end{equation}
for any $\Phi_i=(\bu_i; \bb_i)\in \h^1(\CQ)$. We infer from
\eqref{bo-cont} that $\bo$ is continuous on $\ve\times \ve\times
\ve$. Thus, for any given $\Phi_1, \Phi_2\in \ve\times \ve$ there
exists an element $\mb(\Phi_1, \Phi_2)$ of $\ve^\ast$ such that
\eqref{ex-bilB} holds for any $\Phi_3\in \ve^\ast$.

\noindent \underline{\textit{Proof of item (ii)}:}

\noindent We easily derive from \eqref{7} and \eqref{8} that
\begin{align*}
 \bo(\Phi_1,\Phi_2, \Phi_3)&=0,\\
\bo(\Phi_1, \Phi_2, \Phi_3)&=-\bo(\Phi_1, \Phi_3,
\Phi_2),
\end{align*}
for any $\Phi_i \in \ve, i=1,2,3.$

\noindent \underline{\textit{Proof of item (iii)}:}

\noindent From \eqref{ext-trib-A}-\eqref{ext-trib-B} and the
embedding $\h^2(\CQ)\subset L^{\infty}(\CQ)$ we deduce that
\begin{equation*}
\begin{split}
 |\bo(\Phi_1, \Phi_2, \Phi_3)|\le C \max(1,\mu)\Big[&||\bu_3||_{2,2}\left( ||\bu_2||_1 |\bu_1|+||\bb_2||_2|\bb_1|\right)
\\ &+||\bb_3||_2 \left(||\bu_1||_{2,2}|\bb_2|+||\bu_2||_{2,2}|\bb_1|\right)\Big].
\end{split}
\end{equation*}
By \eqref{eq-bloom} and discrete H\"older's inequality we derive from the last estimate that
\begin{equation*}
 \begin{split}
  |\langle\mb(\Phi_1,\Phi_2), \Phi_3\rangle|\le C\max(1,\mu) ||\phi_3||\Big[|\Phi_1|\,\, ||\Phi_2||+ ||\bu_1||_{2,2}|\bb_2|+||\bu_2||_{2,2}|\bb_1| \Big].
 \end{split}
\end{equation*}
Note that $$ ||\bu_1||_{2,2}|\bb_2|+||\bu_2||_{2,2}|\bb_1| \le ||\bu||_1|\bb_2|+||\bb_1||_2|\bu_2|+||\bu||_1|\bb_1|+||\bb_2||_2|\bu_1|,$$
which with the discrete H\"older inequality yield
$$||\bu_1||_{2,2}|\bb_2|+||\bu_2||_{2,2}|\bb_1| \le ||\Phi_1|||\Phi_2|+||\Phi_2|||\Phi_1|.$$
Consequently,
\begin{equation}\label{ext-trib4}
 \begin{split}
  |\langle\mb(\Phi_1,\Phi_2), \Phi_3\rangle|\le C\max(1,\mu) ||\phi_3||\Big[2 |\Phi_1|\,\, ||\Phi_2||+ ||\Phi_1|||\Phi_2|\Big].
 \end{split}
\end{equation}
From \eqref{ext-trib4} we obtain 
\begin{equation}\label{ext-trib5}
 |\mb(\Phi,\Phi)|_{\ve^\ast}\le 3 C \max(1,\mu) \left(|\Phi| \,\, || \Phi||\right),
\end{equation}
which implies that if $\Phi=(\bu;\bb)$ belongs to $L^2(0,T;\ve)\cap L^\infty(0,T;\h)$ then $\mb(\Phi, \Phi)$ is an element of $L^2(0,T;\ve^\ast)$.
\end{proof}
Thanks to all these preliminary consideration, the problem \eqref{*} can be
rewritten in the following abstract form
\begin{equation}  \label{2}
\begin{cases}
\frac{\partial \by}{\partial t}+ A\by+A_p \by+\mb(\by,\by)=g, \\
\by(0)=\by_0.%
\end{cases}%
\end{equation}
where $\by=(\bu;\bb)$ is a solution of \eqref{*} and
$\by_0=(\bu_0; \bb_0)$. From now on, we will work with \eqref{2}.
The remaining part of this work is devoted to the analysis of
\eqref{2}. In the next two sections we will mainly study the
existence of its weak solutions and their long-time behavior.
\section{Existence of weak solution}
This section is devoted to the investigation of the existence of solutions of \eqref{2}. Before we do so, 
let us define explicitly the concept of solutions that are of interest to us.
\begin{Def}\label{def-weak-sol}
 Let $T>0$ and $\by_0\in H$. A weak solution of \eqref{2} on $[0,T]$ is a function $\by\in L^\infty(0,T;\h)\cap L^2(0,T;\ve)$ such that
\begin{equation*}
 \frac{\partial \by}{\partial t}\in L^2(0,T;\ve^\ast),
\end{equation*}
and
\begin{equation*}
 \by(t)-\by(\tau)+\int_\tau^t \mathcal{A}\by(s) ds+\int_\tau^t \mathcal{A}_p\by(s) ds+\int_\tau^t \mb(\by(s),\by(s)) ds=\int_\tau^t g(s) ds,
\end{equation*}
holds as an equality in $\ve^\ast$ for almost everywhere $t,\tau\in [0,T]$ with $t\ge \tau$.
\end{Def}
If $\by\in L^\infty(0,T;\h)\cap L^2(0,T;\ve)$ and $\frac{\partial \by}{\partial t}\in L^2(0,T;\ve^\ast)$, then $\by\in C(0,T;\ve^\ast)\cap C(0,T;\h_\omega)$, where $\h_\omega$ denotes the
space $\h$ endowed with the weak topology.
Therefore the initial condition $\by(0)=\by_0\in \h$ is meaningful.

The main result of this section is formulated below.
\begin{thm}\label{MAIN-EX}
 Let $\mathbf{T}$ and $\Sigma$
satisfy \eqref{tens1}-\eqref{tens3}, $\by_0\in \h$, $g\in
L^2(0,T;\ve^\ast)$ and $p\in (1,\frac{2n+6}{n+2}]$.
  Then there exists at least one solution of \eqref{2} in the sense of Definition \ref{def-weak-sol}.
\end{thm}
 The proof of this statement is based on a blending of Galerkin approximation and
compactness-monotonicity method. Several steps are needed to achieve the target we are aiming for in the current part of the article.

\underline{\textsc{Step 1: Galerkin Approximation}}

For this step we consider the space $\ve^m$ and the operator $P_m$
described as follows:

\noindent for any positive integer $m$ we set
\begin{equation}\label{Gal-Bas}
 \ve^m= \Span \{\Psi_k=(\phi_k; \psi_k): k=1, 2, 3, \dots, m\},
\end{equation}
and $P_m$ the orthogonal projection from $\ve^\ast$, $\h$ onto
$\ve^m.$ It is well known (see, for instance,
\cite{MALEKetal}) that
\begin{equation}\label{Or-proj}
 |P_m|_{\mathcal{L}(\ve^\ast; \ve^\ast)}\le 1.
\end{equation}
 We look for a
continuous function $\bym(t)$ taking values in $\ve^m$ such that
\begin{equation}\label{Gal-eq1}
\begin{split}
\frac{\partial \bym}{\partial t}(t)+ P_m\left(\mathcal{A}\bym(t)+
\mathcal{A}_p\bym(t) +\mb(\bym(t), \bym(t)) \right)=P_m g,\\
\bym(0)=\bym_0,
\end{split}
\end{equation}
where $\bym_0$ is the orthogonal projection of $\by_0$ with
respect to the scalar product of $\h$ onto $\ve^m$. The system
\eqref{Gal-eq1} is a system of ordinary differential equations
with locally continuous coefficients, thus the existence of a
continuous function $\bym(t)$ on a short interval $[0,T_m]$ is
ensured by Peano's theorem. The global existence will follow from
the a priori estimates. As mentioned in the introduction $C$ will
describe positive constants depending only on the data (not on
$m$) and which may change from one term to the next.

\underline{\textsc{Step 2: Derivation of a priori estimates}}

To prove the compactness of our Galerkin solution which will allow us to pass to the limit in \eqref{Gal-eq1}, we need to derive
several crucial estimates.

\noindent Multiplying \eqref{Gal-eq1} by $\bym(t)$ yields
\begin{equation*}
\begin{split}
\left(\frac{\partial \bym(t)}{\partial t}, \bym(t)\right)+ \langle
P_m(\mathcal{A}\bym(t)+\mathcal{A}_p\bym(t),
\bym(t))\rangle=\langle P_m g(t),\bym(t)\rangle,
\end{split}
\end{equation*}

where we have used the fact that
\begin{equation*}
\begin{split}
\langle P_m
\mb(\bym(t),\bym(t)),\bym(t)\rangle=\langle\mb(\bym(t),\bym(t)),
\bym(t)\rangle,\\
=0 \quad (\text{ thanks to \eqref{10}}.)
\end{split}
\end{equation*}
From this we infer that
\begin{equation}\label{GAL-EQ1}
\frac{1}{2}\frac{d|\bym(t)|^2}{d
t}+||\bym(t)||^2+\langle\mathcal{A}_p\bym(t),\bym(t)\rangle=\langle
g(t),\bym(t)\rangle.
\end{equation}
Owing to \eqref{4} we obtain from the last estimate that
\begin{equation*}
\frac{1}{2}\frac{d}{dt} |\bym(t)|^2 + ||\bym(t)||^2\le
||g(t)||_{\ve^\ast}||\bym(t)||.
\end{equation*}
Using Young's inequality yields
\begin{equation*}
\frac{d}{dt} |\bym(t)|^2 + ||\bym(t)||^2\le ||g(t)||^2_{\ve^\ast}.
\end{equation*}
Equivalently,
\begin{equation*}
|\bym(t)|^2 + \int_0^t||\bym(s)||^2ds \le |\bym_0|^2+\int_0^t
||g(s)||^2_{\ve^\ast}ds.
\end{equation*}
Owing to \eqref{Or-proj} we have 
\begin{equation}\label{Gal-eq2}
|\bym(t)|^2 + \int_0^t||\bym(s)||^2ds \le C(\by_0, g, T),
\end{equation}
where
$$ C(\by_0, g, T)=|\by_0|^2+\int_0^T g(t) dt,$$
for any $m>0$. Since the constant $C(\by_0, g, T)$ does not depend
on $m$ we have $T_m=T$.

\underline{\textsc{Step 3: Passage to the limit}}

Thanks to \eqref{Gal-eq2} and part (ii) of Lemma \ref{lem1} and
part (iii) of Lemma \ref{lem2}, we can derive that $\frac{\partial
\bym}{\partial t}$ belongs to a bounded set of
$L^2(0,T;\ve^\ast)$.
By a diagonal process we can find
a subsequence of $\bym$ which is not relabelled and a function
$\by$ such that
\begin{align}
\bym\rightarrow \by \text{ weakly-star in } L^\infty(0,T;\h),\label{Gal-eq3}\\
\frac{\partial \bym}{\partial t} \rightarrow \frac{\partial
\by}{\partial t} \text{ weakly in }
L^2(0,T;\ve^\ast),\label{Gal-eq5} \\
\bym\rightarrow \by
\text{ weakly in } L^2(0,T;\ve).\label{Gal-eq4}
\end{align}
Furthermore, by applying the compactness result in \cite[Lemma 5.1]{LIONS} we
can check that
\begin{equation}\label{Gal-eq6}
\bym\rightarrow \by \text{ strongly in } L^2(0,T;\h).
\end{equation}
Also, owing to \eqref{Gal-eq5} we see that
\begin{equation}\label{Gal-eq7}
\bym(T)\rightarrow \by(T) \text{ weakly in } \h.
\end{equation}
Now to complete the proof of the existence of weak solution we
have to pass to the limit in \eqref{Gal-eq1}.

Since $\mathcal{A}$ is a continuous linear mapping from
$L^2(0,T;\ve)$ into $L^2(0,T;\ve^\ast)$, then we have that
\begin{equation}\label{Gal-eq8}
P_m\mathcal{A}\bym \rightarrow \mathcal{A}\by \text{ weakly in }
L^2(0,T;\ve^\ast).
\end{equation}
Thanks to part (iii) of Lemma \ref{lem2} and \eqref{Gal-eq2}
$P_m\mb(\bym,\bym)$ belongs to a bounded set of
$L^2(0,T;\ve^\ast)$. Taking advantage of \eqref{Gal-eq6} and
\eqref{Gal-eq4}, we will show that
\begin{equation}\label{Gal-eq9}
P_m\mb(\bym,\bym)\rightarrow \mb(\by,\by) \text{ weakly in }
L^2(0,T;\ve^\ast).
\end{equation}
To this end let
$$\mathbb{D}=\{\Phi=\chi(t)\Psi_k: \chi(t)\in \mathbb{C}^\infty_c(0,T)\text{ and } k=1,2,\dots\},$$
where $\{\Psi_k; k=1,2, \dots\}$ is defined in \eqref{Gal-Bas}. It
is clear that this set is dense in $L^2(0,T;\ve)$.  Owing to
\cite[Proposition 21.23]{Zeidler-2A}, the claim \eqref{Gal-eq9} is
achieved if we prove that
\begin{equation*}
\int_0^T \langle\mb(\byms,\byms)-\mb(\by,\by), \Psi_k\rangle
\chi(s) ds\rightarrow 0,
\end{equation*}
for any $\Phi=\chi(t)\Psi_k\in \mathbb{D}$.  For this purpose, we
rewrite the last identity in the following form
\begin{equation*}
\begin{split}
\int_0^T \langle\mb(\byms,\byms)-\mb(\by,\by),
\Psi_k\rangle\chi(s) ds=I_1+I_2.
\end{split}
\end{equation*}
where
\begin{align*}
I_1=\int_0^T\langle \mb(\byms-\by(s),\byms),\Psi_k\rangle
\chi(s)ds,\\
I_2=\int_0^T \langle \mb(\by(s),\by(s)-\bym(s)),\Psi_k\rangle
\chi(s)ds.
\end{align*}
 For fixed $\Phi$ and $\Xi$ the mapping
$\Upsilon\mapsto \int_0^T \langle \mb(\Xi,\Upsilon),\Psi_k\rangle
\chi(s)ds$ is a continuous linear functional on $L^2(0,T;\ve)$.
Hence by invoking \eqref{Gal-eq4} $I_2$ converges to 0 as
$m\rightarrow \infty$. Next, we easily derive from
\eqref{ext-trib-A} that
\begin{equation*}
\begin{split}
\left\lvert\int_0\langle \mb(\byms-\by(s),\byms),\Psi_k\rangle
\chi(s)ds\right\rvert\le C \left(\int_0^T |\byms-\by(s)| ||\byms||ds\right)\\
\times ||\Phi||_{L^\infty([0,T]\times \mathcal{Q})},
\end{split}
\end{equation*}
which together with H\"older's inequality and \eqref{Gal-eq2}
imply that
\begin{equation*}
\begin{split}
\left\lvert\int_0\langle \mb(\byms-\by(s),\byms),\Psi_k\rangle
\chi(s)ds\right\rvert\le C \left(\int_0^T
|\byms-\by(s)|^2ds\right)^\frac{1}{2}\\
\times ||\Phi||_{L^\infty([0,T]\times \mathcal{Q})}.
\end{split}
\end{equation*}
Thanks to \eqref{Gal-eq6} the left hand side of this last
inequality will converge to 0 as $m\rightarrow \infty$. Hence we
have prove that $I_1$ converges to 0 as $m\rightarrow \infty$
which also shows that \eqref{Gal-eq9} holds.

 In view of part (ii) of Lemma \ref{lem1} and
\eqref{Gal-eq2}, the nonlinear term $\mathcal{A}_p\bym$ is an
element of a bounded set of $L^2(0,T;\ve^\ast)$. Therefore there
exists an element $\Omega$ of $L^2(0,T;\ve^\ast)$ such that
\begin{equation}\label{Gal-eq10}
\mathcal{A}_p\bym \rightarrow \Omega \text{ weakly in }
L^2(0,T;\ve^\ast),
\end{equation}
as $m$ approaches $\infty$. It remains to be shown that
\begin{equation}\label{Gal-eq11}
\Omega=\mathcal{A}_p \by \text{ in } L^2(0,T;\ve^\ast).
\end{equation}
In vertu of \eqref{Gal-eq5}-\eqref{Gal-eq10}, we can pass to the
limit in \eqref{Gal-eq1} and find that $\by$ satisfies
\begin{equation*}
\frac{\partial \by}{\partial
t}+\mathcal{A}\by+\Omega+\mb(\by,\by)=g
\end{equation*}
as an equality in $L^2(0,T;\ve^\ast)$. Therefore by an integration
by parts we see that
\begin{equation}\label{55}
\int_0^T \langle\Omega,\by(s)\rangle ds=\int_0^T
\langle\Sigma(\by(s)),\by(s)\rangle
ds-\frac{1}{2}(|\by(T)|^2-|\by(0)|^2),
\end{equation}
where the functional $\Sigma$ is defined by
\begin{equation*}
\Sigma(\Phi)=g-\mathcal{A}\Phi-\mb(\Phi,\Phi).
\end{equation*}
Since $\mathcal{A}_p$ is monotone, then for any $\Phi\in
L^2(0,T;\ve)$ there holds
\begin{equation*}
\langle\mathcal{A}_p\bym,\bym\rangle\ge \langle\mathcal{A}_p\bym,
\Phi\rangle+\langle \mathcal{A}_p \Phi, \bym-\Phi\rangle.
\end{equation*}
But from \eqref{Gal-eq1} we see that
\begin{equation*}
\langle\mathcal{A}_p\bym,\bym\rangle=\int_0^T\langle
\Sigma(\byms), \byms\rangle-\frac{1}{2}(|\bym(T)|^2-|\bym(0)|^2),
\end{equation*}
which along with the former estimate and \eqref{55} yield (after
taking the limit as $m\rightarrow \infty$) that
\begin{equation}\label{231}
\int_0^T \langle \Omega-\mathcal{A}_p\Phi(s),
\by(s)-\Phi(s)\rangle ds\ge 0.
\end{equation}
At this juncture, we let $\Phi=\by-\beta \Psi$, for any $\beta>0$
and $\Psi\in L^2(0,T;\ve)$. Thus after division by $\beta$ and
letting $\beta\rightarrow 0$, \eqref{231} leads to
\begin{equation*}
\int_0^T \langle \Omega-\mathcal{A}_p \by, \Psi\rangle
ds\ge0,\,\,\,  \forall \Psi\in L^2(0,T;\ve),
\end{equation*}
from which the sought convergence
\eqref{Gal-eq10}-\eqref{Gal-eq11} follows. Hence thanks to the
convergence \eqref{Gal-eq3}, \eqref{Gal-eq5}, \eqref{Gal-eq4},
\eqref{Gal-eq8}, \eqref{Gal-eq9} and \eqref{Gal-eq11} we see after
passage to the limit that there exists at least a weak solution
(in sense of Definition \ref{def-weak-sol}) $\by$ to \eqref{2}.
\section{Trajectory attractor of the solutions of \eqref{2}}
In this section we study the behavior of a weak solution $\by$ of
\eqref{2} for large time. We mainly investigate the existence of
an attractor \`a la Chepyzhov and Vishik. For this purpose we
closely follow the presentation in \cite{VISHIK+CHEPYZHOV} which
is the pioneering article dealing with the uniform trajectory
attractor of non-autonomous evolutionary nonlinear partial
differential equations. This method does not require that a
uniqueness of the solution holds. Hereafter, for a Banach space
$E$ we set
\begin{equation*}
 L^p_b(\rp;E)=\{\bu: \bu(s) \in \Ll^p(\rp;E), \sup_{t\ge 0}\int_t^{t+1} \bu(s) ds<\infty\},
\end{equation*}
 and
\begin{equation*}
 ||\bu||^p_{L^p_b(\rp;E)}=\sup_{t\ge 0}\int_t^{t+1} ||\bu(s)||_E^p ds,
\end{equation*}
for any $\bu\in L^p_b(\rp;E)$, $p\in [1,\infty]$. For any $\tau,
t\ge 0$ with $t\ge \tau$ we define the space $\mathcal{F}_{loc}^+$
by setting
\begin{equation*}
 \mathcal{F}_{loc}^+=\left\{\bu: \bu\in \Ll^2(\rp;\ve)\cap \Ll^\infty(\rp; \h), \frac{\partial \bu}{\partial t} \in \Ll^2(\rp; \ve^\ast)\right\},
\end{equation*}
which is endowed with the metric defined by
\begin{equation*}
 ||\bu||_{\mathcal{F}_{loc}^+}=||\bu||^2_{\Ll^2(\rp ;\ve)}+||\bu||^2_{\Ll^\infty(\rp; \h)}+\left|\left|\frac{\partial \bu}{\partial t}\right|\right|^2_{\Ll^2(\rp; \ve^\ast)},
\end{equation*}
for any $\bu\in \mathcal{F}_{loc}^+$.

\begin{Def}\label{topo-F}
By $\Theta_+^{loc}$ we mean the space $\mathcal{F}_{loc}^+$ endowed with the following convergence topology:

A sequence $\bu^m\subset \mathcal{F}_{loc}^+$ converges to $\bu$ in $\Theta_+^{loc}$ as $m\rightarrow \infty$ if
\begin{align*}
 &\bu^m\rightarrow \bu \text{ weakly in } L^2(\tau, t; \ve),\\
  &\bu^m\rightarrow \bu \text{ weakly-star in } L^\infty(\tau, t; \h),\\
&\frac{\partial \bu^m}{\partial t} \rightarrow \frac{\partial \bu}{\partial t} \text{ weakly in } L^2(\tau, t; \ve^\ast),
\end{align*}
for any compact interval $[\tau,t]\subset \rp$.
\end{Def}


We define $\mathcal{F}^a_{+}$ as the space of functions $\bu$ defined on $\rp$ such that
\begin{equation*}
 ||\bu||_{\mathcal{F}^a_{+}}=||\bu||^2_{L^2_{b} (\rp ;\ve)}+||\bu||^2_{L^\infty_{b} (\rp; \h)}+\left|\left|\frac{\partial \bu}{\partial t}\right|\right|_{L^2_{b} (\rp; \ve^\ast)},
\end{equation*}
 is finite.
Note that $\mathcal{F}^a_{+}$ is a Banach space with norm $||\bu||_{\mathcal{F}^a_{+}}$ and any bounded set of $\mfa$ is precompact in $\Theta^{loc}_+$.
Let $g_0\in \Ll^2(\rp;\ve^\ast)$ be an external force which is translation bounded, that is
\begin{equation*}
 ||g_0||^2_{L^2_b(\rp;\ve^\ast)}<\infty.
\end{equation*}
 \begin{Def}
  The synbol space $\Sigma(g_0)$ is the closure of the set $\{g_0(t+\cdot), t\ge 0\}$ in $\Ll^2(\rp,\ve^\ast)$.
 \end{Def}
From now on, we denote by $\Llw^2(\rp;\ve^\ast)$ the space
$\Ll^2(\rp;\ve^\ast)$ equipped with the weak topology. Under some
constraints on $g_0$ the space $\Sigma(g_0)$ enjoys some
properties which will be important here. We state them
below.
\begin{lem}\label{lem-traj}
 If $g_0\in \Ll^2(\rp;\ve)$ is translation bounded then $\Sigma(g_0)$ is a compact and complete metric space.
\end{lem}
\begin{proof}
 If $g_0\in \Ll^2(\rp;\ve)$ is translation bounded the it follows from \cite[Proposition 6.8]{VISHIK+CHEPYZHOV} that $g_0$ is translation compact in $\Llw^2(\rp;\ve^\ast)$, that is
the closure of $\{g_0(t+\cdot), t\ge 0\}$ in $\Ll^2(\rp,\ve^\ast)$ is compact in $\Llw^2(\rp;\ve^\ast)$. Denote this closure by $\mathcal{H}(g_0)$. From
\cite[Lemma 6.3]{VISHIK+CHEPYZHOV}, $\mathcal{H}(g_0)$ as a topological subspace of $\Llw^2(\rp;\ve^\ast)$ is metrizable and the corresponding metric space is complete. Hence
$\mathcal{H}(g_0)$ can be viewed as a compact and complete metric space. But from \cite[Section 8]{VISHIK+CHEPYZHOV}, $\mathcal{H}(g_0)$ coincides with $\Sigma(g_0)$ up to a homeomorphism.
Therefore, $\Sigma(g_0)$ as a topological subset of $\Ll^2(\rp;\ve^\ast)$ is a compact and complete metric space.
\end{proof}
\begin{Rem}
Before we carry on, we should note that the result of the  Lemma \ref{lem-traj} holds true if we consider $\mathbb{R}$ in place of $\rp$.
\end{Rem}
\begin{Def}
For a translation bounded function $g\in L^2_{loc}(\rp; \ve^\ast)$ we define the trajectory space $\mathcal{K}^{+}_{g}$ as the set of weak solutions
of \eqref{2} on any compact interval $I$ such that the energy inequality
\begin{equation}\label{NRJ}
 e^{\alpha t}|\by(t)|^2-e^{\lambda \tau}|\by(\tau)|^2+\int_\tau^t e^{\lambda s}\left[||\by(s)||^2-\lambda |\by(s)|^2\right]\le \frac{\Omega_\lambda(g)}{\lambda}
(e^{\lambda t}-e^{\lambda \tau}),
\end{equation}
holds for almost every $\tau,t\in I$ with $t\ge \tau+1$ and
\begin{equation*}
 \Omega_\lambda(g)=\sup_{h\in [1,2]}\sup_{t\ge0}\left(\frac{\lambda\int_0^h ||g(s+t)||^2_{\ve^\ast}e^{\lambda s} ds}{e^{\lambda h}-1}\right).
\end{equation*}
\end{Def}
Now let $$\mck_{\Sigma(g_0)}=\cup_{g\in \Sigma(g_0)}\mck_{g}^{+}.$$ The space $\mck_{\Sigma(g_0)}$ is a topological space equipped with the topology of
$\Theta^{loc}_{+}$. We recall some crucial properties of $\mck_{g}^{+}$ and $\mck_{\Sigma(g_0)}$.
%
\begin{lem}\label{lem-traj2}
\begin{enumerate}
 \item[(i)] For any $\by_0\in \h$ and a translation bounded function $g \in \Ll^2(\rp;\ve^\ast)$, $\mck^{+}_g$ is non-empty and the following energy inequality holds
\begin{equation}\label{NRJ-C}
 e^{\lambda t}|\by(t)|^2-e^{\lambda \tau}|\by(\tau)|^2+\int_\tau^t e^{\lambda s}\left[||\by(s)||^2-\lambda |\by(s)|^2\right]\le \frac{\Omega_\lambda(g)}{\lambda}
(e^{\lambda t}-e^{\lambda \tau}),
\end{equation}
 for almost every $\tau,t\in \rp$ with $t\ge \tau+1$.
\item[(ii)] The set $\mck^+_g$ is closed in $(\thl,\Sigma(g_0))$ endowed with the usual topology of cartesian product space,  and
$\mck_{\Sigma(g_0)}$ is closed in $\Theta^{loc}_{+}.$
\end{enumerate}
\end{lem}
Before we proceed to the proof of Lemma \ref{lem-traj2} we recall
the following result which is taken from \cite[Lemma
9.2]{VISHIK+CHEPYZHOV}.
\begin{lem}\label{lem-ViC}
Let $f, \chi\in \Ll^1(\rp)$ and assume that the following inequality holds:
\begin{equation*}
 -\int_0^\infty \Phi^\prime(s) f(s) ds+\eta \int_0^\infty f(s) \Phi(s) ds\le \int_0^\infty \chi(s) \Phi(s) ds,
\end{equation*}
for any $\Phi\in C^\infty_0(\rp;\rp)$ and for some $\eta\in \mathbb{R}$. Then
\begin{equation*}
 f(t)e^{\eta t}-f(\tau)e^{\eta \tau}\le \int_\tau^t \chi(s) e^{\eta s} ds,
\end{equation*}
for almost all $t,\tau\in \rp$ with $t\ge\tau$.
\end{lem}
\begin{proof}[Proof of Lemma \ref{lem-traj2}]
Let $\bym_0$ be the projection of $\by_0$ on the space $\ve^m$ and let $\bym$ be a sequence of Galerkin solutions of \eqref{2}. We have seen from the proof of Theorem \ref{MAIN-EX} that
for any $\by_0\in \h$ and $g\in \Ll^2(\rp;\ve^\ast)$ there exists at least one solution to \eqref{2}. Now it remains to prove the energy inequality. To do so
let us consider an element $\Phi\in C^\infty_0(\tau, T)$ with
$\Phi\ge 0$. We have 
\begin{equation}\label{Gal-EQ2}
\int_\tau^T \Phi^\prime(t)\left||\bym(t)|-|\by(t)|\right|^2dt \le
\int_\tau^t \Phi^\prime(t)|\bym(t)-\bym|^2 dt,
\end{equation}
which implies that $\sqrt{\Phi^\prime}|\bym|$ converges to
$\sqrt{\Phi^\prime}|\by|$ strongly in $L^2(\tau,T)$. hence, we can
extract a subsequence from $\sqrt{\Phi^\prime}|\bym|$ (not
relabelled) such that it converges to $\sqrt{\Phi^\prime}|\by|$
almost everywhere in $[\tau,T]$.  Since $\Phi^\prime(t)$ is
continuous on a compact set of $\mathbb{R}$ then it follows from
\eqref{Gal-eq2} that the sequence of numerical functions
$\Phi^\prime(t)|\bym(t)|^2$ is uniformly bounded. Thus we derive
from \eqref{Gal-EQ2} and the Lebesgue Dominated Convergence Theorem
that
\begin{equation}\label{Nrj-eq1}
-\int_\tau^t \Phi^\prime(s) |\bym(s)|^2 ds\rightarrow -\int_\tau^t
\Phi^\prime(s) |\by(s)|^2 ds.
\end{equation}
As $\sqrt{\Phi(s)}$ is uniformly bounded on $[\tau, t]$, therefore
it follows from \eqref{Gal-eq2} that $\sqrt{\Phi(s)}\bym(s)$
belongs to a bounded set of $L^2(\tau,t; \ve)$ and
\begin{equation*}
\sqrt{\Phi(s)}\bym(s)\rightarrow \sqrt{\Phi(s)}\by(s) \text{
weakly in } L^2(\tau, t; \ve),
\end{equation*}
as $m\rightarrow \infty$. Moreover,
\begin{equation}\label{Nrj-eq2}
\int_\tau^t \Phi(s) ||\by(s)||^2 ds\le {\lim \inf}_{m\rightarrow
\infty}\int_\tau^t \Phi(s) ||\bym(s)||^2 ds.
\end{equation}
By taking a test function $\Phi\in C^\infty_0(\tau, t)$ with
$\Phi\ge 0$, we obtain from \eqref{GAL-EQ1} that
\begin{equation*}
\begin{split}
-\frac{1}{2}\int_\tau^t \Phi^\prime (s) |\bym(s)|^2ds +\int_\tau^t
\Phi(s)||\bym(s)||^2 ds +\int_\tau^t \Phi(s)
\langle\mathcal{A}_p\bym(s),\bym(s)\rangle ds \\=\int_\tau^t \Phi(s)
\langle g(s),\bym(s)\rangle ds.
\end{split}
\end{equation*}
Taking into account \eqref{4} and using Young's inequality we get
from the last estimate that
\begin{equation}\label{Nrj-eq4}
-\int_\tau^t \Phi^\prime (s) |\bym(s)|^2ds + \int_\tau^t
\Phi(s)||\bym(s)||^2 ds \le  \int_\tau^t
\Phi(s)||g(s)||^2_{\ve^\ast}ds.
\end{equation}
Owing to \eqref{Nrj-eq1} and \eqref{Nrj-eq2} we pass to the limit
in \eqref{Nrj-eq4} and obtain the following energy inequality
\begin{equation*}
 -\int_\tau^t \Phi^\prime (s) |\by(s)|^2ds + \int_\tau^t
\Phi(s)||\by(s)||^2 ds \le  \int_\tau^t
\Phi(s)||g(s)||^2_{\ve^\ast}ds,
\end{equation*}
which can be rewritten in the following form
\begin{equation*}
\begin{split}
-\int_\tau^t \Phi^\prime (s) |\bym(s)|^2ds + \lambda \int_\tau^t
\Phi(s)|\bym(s)|^2 ds \\
\le  \int_\tau^t
\Phi(s)\left(||g_0(s)||^2_{\ve^\ast}-[||\by(s)||^2-\lambda
|\by(s)|^2]\right)ds.
\end{split}
\end{equation*}

This estimate along with Lemma \ref{lem-ViC} yields
\begin{equation}
|\by(t)|^2e^{\lambda t}-|\by(\tau)|^2e^{\lambda \tau}+
\int_\tau^{t}\left(||\by(s)||^2-|\by(s)|^2\right)e^{\lambda
s}ds\le \int_\tau^{t} ||g(s)||^2_{\ve^\ast} e^{\lambda s} ds,
\end{equation}
for almost all $\tau, t\in \rp$ with $t\ge \tau$. To end the proof
of item (i) we can use the same argument as in \cite[Corollary
9.4]{VISHIK+CHEPYZHOV}.

Item (ii) can be shown by using the same argument as in
\cite[Proposition 8.5]{VISHIK+CHEPYZHOV} (see also
\cite[Proposition 8.3]{VISHIK+CHEPYZHOV}).
\end{proof}
We define a semi-group $\{\mathbb{S}_t, t\ge 0\}$ acting on $\Sigma(g_0)$ and $\mck_{\Sigma(g_0)}$ by
\begin{equation}
 \mathbb{S}_t\bu(\cdot)=\bu(t+\cdot), t\ge 0.
\end{equation}
Important facts concerning its action on $\Sigma(g_0)$ and $\mck_{\Sigma(g_0)}$ are stated in the following lemma.
\begin{lem}\label{lem-traj3}
 Assume that $g_0$ is translation bounded in $\Ll^2(\rp;\ve^\ast)$ (i.e., $g_0$ is translation compact in $\Llw^2(\rp; \ve^\ast)$). Then
\begin{enumerate}
 \item[(i)] $\mathbb{S}_t$ is continuous on $\Sigma(g_0)$ in the
topology of $\Ll^2(\rp;\ve^\ast)$.
\item[(ii)] We have
$\mathbb{S}_t\Sigma(g_0)\subset \Sigma(g_0)$ for any $t\ge 0$.
\item[(iii)] We have
\begin{equation}
 \mathbb{S}_t\mck^+_g\subset \mfa \subset \mfl, \forall g\in \Sigma(g_0).
\end{equation}
\item[(iv)] For any $g\in \Sigma(g_0)$  we have
\begin{equation}\label{tr-coo}
 \mathbb{S}_t\mck_g^+\subset \mathcal{K}^+_{\mathbb{S}_t g},
\end{equation}
which implies that
\begin{equation}
 \mathbb{S}_t\mck_{\Sigma(g_0)}\subset \mck_{\Sigma(g_0)}, t\ge 0.
\end{equation}
\end{enumerate}
\end{lem}
\begin{proof}
Proofs of (i) and (ii) are straightforward.

Item (iii) is a direct consequence of \eqref{NRJ-C}. Item (iv) is
proved as follows: If $g\in \Sigma(g_0)$ and $\by \in \mck_g^+$,
then $\by(t+\cdot), t\ge 0$ is a solution of \eqref{2} with
initial condition $\by(t)$ and external force $g(t+\cdot)$; that
is, $\mathbb{S}_t\by \in \mck_{\mathbb{S}_t g}$.
\end{proof}

To proceed with our investigation we introduce the concept of
attractor of our interest. The definitions we state below are
taken from \cite{VISHIK+CHEPYZHOV}.
\begin{Def}
A set $P\subset \Theta^{loc}_{+}$ is a uniformly (wrt to $g\in
\Sigma(g_0)$) attracting set for the family $\{\mck_g, g\in
\Sigma(g_0)\}$ in the topological space $\Theta^{loc}_{+}$ if for
any bounded set $\mathbb{B}$ of $\mathcal{F}^a_{+}$ and
$\mathbb{B}\subset \mck_{\Sigma(g_0)}^+$ the set $P$ attracts
$\mathbb{S}_t \mathbb{B}$ as $t\rightarrow \infty$; that is, for
any neighborhood $\mathcal{O}(P)$ in $\Theta^{loc}_+$ there exists
$t_1>0$ such that $\mathbb{S}_t\mathbb{B}\subset \mathcal{O}(P)$
for any $t\ge t_1$.
\end{Def}
\begin{Def}
 A set $\mathcal{A}_{\Sigma(g_0)}$ is a uniform (wrt to $g\in \Sigma(g_0)$) trajectory attractor for the
 family $\{\mck_g, g\in \Sigma(g_0)\}$ in the topological space $\Theta^{loc}_{+}$ if
 \begin{enumerate}
\item $\mathcal{A}_{\Sigma(g_0)}$ is compact,

\item $\mathcal{A}_{\Sigma(g_0)}$ is strictly invariant, that is,
$\mathbb{S}_t \mathcal{A}_{\Sigma(g_0)}=\mathcal{A}_{\Sigma(g_0)}$
for any $t\in \rp$,

\item and $\mathcal{A}_{\Sigma(g_0)}$ is a minimal uniformly
attracting set for $\{\mck^{+}_g, g\in \Sigma(g_0)\}$.
 \end{enumerate}
\end{Def}

Before we state the main result of this section we still need to introduce the concept of $\omega$-limit set.
\begin{Def}
 The $\omega$-limit set of $\Sigma(g_0)$ is defined through
\begin{equation}\label{omega-lim}
 \omega(\Sigma(g_0))= \bigcap_{t\ge 0}\Big[\bigcup_{h\ge t}\mathbb{S}_h \Sigma(g_0) \Big]_{\Sigma(g_0)},
\end{equation}
where $[\,\,\cdot\,\,]_{\Sigma(g_0)}$ designates the closure in $\Sigma(g_0)$.
\end{Def}
For any translation bounded function $g_0\in\Ll^2(\rp;\ve^\ast)$,
Lemma \ref{lem-traj} implies that the metric space $\Sigma(g_0)$
is compact and complete. Hence owing to the item (i) of Lemma
\ref{lem-traj3} and well-known theorem on continuous semigroup
acting on compact and complete metric space (see for example
\cite{BABIN+VISHIK}, \cite{TEMAM-INF} ), $\mathbb{S}_t$ as a
continous semigroup acting on $\Sigma(g_0)$ has a global attractor
$\mathbb{A}$ in $\Sigma(g_0)$. The set $\mathbb{A}$ coincides with
$\omega(\Sigma(g_0))$ which implies that
\begin{equation}\label{inv-ga}
 \mathbb{S}_t\omega(\Sigma(g_0))=\omega(\Sigma(g_0)),
\end{equation}
 for any $t\ge 0$.

From \eqref{omega-lim} we clearly see that $\omega(\Sigma(g_0))\subset \Sigma(g_0)$, so it makes sense to define the analog of
$\{\mck_g^+, g\in \Sigma(g_0)\}$ on $\omega(\Sigma(g_0)$ by taking  $g$ in the smaller space $\omega(\Sigma(g_0))$ instead of in $\Sigma(g_0)$. Therefore, we can also define
 the notion of uniform attractor for $\{\mck_g^+, g\in \omega(\Sigma(g_0))\}$ with respect to $g\in \omega(\Sigma(g_0))$.

Now we formulate the main result of this section.
\begin{thm}\label{TRAJ-THM}
Let $g_0$ be translation bounded in $\Ll^2(\rp; \ve^\ast)$ (that
is, it is translation compact in $\Llw^2(\rp; \ve^\ast)$), then
the translation semigroup $\{\mathbb{S}_t, t\ge 0\}$
 acting on $\mck_{\Sigma(g_0)}^+$ possesses a uniform trajectory attractor $\mathcal{A}_{\Sigma(g_0)}$. The set $\mathcal{A}_{\Sigma(g_0)}$
 is bounded in $\mathcal{F}^a_+$ and compact in $\Theta^{loc}_+$. It is strictly invariant:
\begin{equation}
 \mathbb{S}_t\mathcal{A}_{\Sigma(g_0)}=\mathcal{A}_{\Sigma(g_0)},
\end{equation}
for all $t\ge 0$. Furthermore, we have
 \begin{equation}
\mathcal{A}_{\Sigma(g_0)}=\mathcal{A}_{\omega(\Sigma(g_0))},
 \end{equation}
 where $\mathcal{A}_{\omega(\Sigma(g_0))}$ is the uniform
 (with respect to $g\in \omega(\Sigma(g_0))$) attractor of $\{\mck_g^+, g\in \omega(\Sigma(g_0) \}$.
\end{thm}

As in classical theory of dynamical system of PDEs, to show the
existence of a trajectory attractor it is crucial to find a
uniform attracting set for $\{\mathbb{S}_t, t\ge 0\}$ which is
bounded in $\mfa$ and compact in $\thl$. Since $\thl$ is a
topological space endowed with a weak topology defined in
Definition \ref{topo-F} it is enough to construct a bounded set in
$\mfa$ which will attract $\mathbb{S}_t \mathbb{B}$ for any
bounded set $\mathbb{B}\subset \mck_{\Sigma(g_0)}^+$. For the
problem \eqref{*}, the following estimate is the main tool for
such a construction.
\begin{lem}\label{TRAJ-lem}
 Assume that $g_0$ is translation compact in $\Llw^2(\rp; \ve^\ast)$ and let $q=2p-2$. Then there exist positive constants $R_0$ and $C_0$ depending only
on the data of \eqref{*} such that
\begin{equation}
 ||\mathbb{S}_t \by||_{\mathcal{F}^a_{+}}\le C_0 ||\by||^2_{L^\infty(0,1; \h)}e^{-\lambda t}+R_0, \forall t\ge 1, \,\,\, \text{ if } p\in (1,2],
\end{equation}
 and
\begin{equation}\label{75}
 ||\mathbb{S}_t \by||_{\mathcal{F}^a_{+}}\le C_0 ||\by||^q_{L^\infty(0,1; \h)}e^{-\frac{q \lambda t}{2}}+R_0, \forall t\ge 1,
\end{equation}
if $p\in (2, \frac{2n+6}{n+2}]$.
\end{lem}
\begin{proof}[\textbf{Proof of Lemma \ref{TRAJ-lem}}]
We distinguish the case $p\in (1,2]$ and $p\in (2, \frac{2n+6}{n+2}]$.

\noindent \underline{\textsc{Case I: $p\in (1,2]$} }

\noindent The proof will be split in several steps. Throughout, $\by$ will denote an element of $\mck_{(\Sigma(g_0))}$ and $g\in \Sigma(g_0)$.

\noindent \texttt{Estimate for $\mathbb{S}_t\by$ in $L^\infty_b(\rp; \h)$}

Owing to the Poincar\'e'-like inequality \eqref{Poincare} we derive from \eqref{NRJ} that
\begin{equation}\label{NRJ-B}
 e^{\lambda t}|\by(t)|^2-e^{\lambda \tau}|\by(\tau)|^2\le \frac{\Omega_\lambda(g)}{\lambda}
(e^{\lambda t}-e^{\lambda \tau}),
\end{equation}
holds for almost every $\tau,t\in I$ with $t\ge \tau+1$.  Recall that for a fucntion $g_0$ which is translation bounded in $\Ll^2(\rp;\ve^\ast)$, the following holds (see \cite[Section 8]{VISHIK+CHEPYZHOV})
$$||g||^2_{L^2_b(\rp;\ve^\ast)}\le  ||g_0||^2_{L^2_b(\rp;\ve^\ast)},$$
which combined with \eqref{NRJ-B} yields
\begin{equation*}
 |\by(t+s)|^2\le |\by(s-t)|^2 e^{-\lambda t}+ \frac{\Omega_\lambda(g_0)}{\lambda }(1-e^{-\lambda t}),
\end{equation*}
for $t\ge 1$ and $s\ge 2$. Therefore, we obtain that
\begin{equation}\label{att-set1}
 ||\mathbb{S}_t \by||^2_{L^\infty_b(\rp; \h)}\le ||\by||^2_{L^\infty(0,1;\h)}e^{-\lambda t}+\frac{\Omega_\lambda(g_0)}{\lambda },
\end{equation}
for any $t\ge 1$.

\noindent \texttt{Estimate for $\mathbb{S}_t\by$ in $L^2_b(\rp;\ve)$}

Hereafter let us set $R_1=\frac{\Omega_\lambda(g_0)}{\lambda }$.
By integrating \eqref{NRJ-C} between $t$ and $t+1$ and invoking
\eqref{att-set1}, we obtain that
\begin{equation}\label{att-set2}
 \lambda \int_t^{t+1} |\by(s)|^2 ds\le \lambda ||\by||^2_{L^\infty(0,1;\h)}+ R_1 e^{\lambda t} (e^\lambda -1),
\end{equation}
which along with \eqref{NRJ-C} and \eqref{att-set1} implies that
\begin{equation}\label{att-set3}
 \int_t^{t+1} \left(||\by(s)||^2- \lambda |\by(s)|^2\right)e^{\lambda s} ds\le \lambda R_1 e^{\lambda t} (e^\lambda -1)+ R_1 e^{\lambda t} + ||\by||_{L^\infty(0,1;\h)}^2.
\end{equation}
By vertu of \eqref{att-set2} and \eqref{att-set3} we have 
\begin{equation}\label{att-set4}
 \int_t^{t+1}||\by(s)||^2 e^{\lambda s}ds\le 2 R_1(e^\lambda-1)e^{\lambda t}+ R_1 e^{\lambda t }+(1+\lambda)||\by||^2_{L^\infty(0,1;\h)}.
\end{equation}
For $s\in [t,t+1]$, we can easily check that $e^{\lambda t}
||\by(s)||^2\le e^{\lambda s} ||\by(s)||^2$. Hence, we infer from
\eqref{att-set4} the existence of positive constants $R_2$ and
$C_2$ such that that
\begin{equation}\label{att-set5}
 ||\mathbb{S}_t\by||^2_{L^2_b(\rp;\ve)}\le R_2+C_2 e^{-\lambda t} ||\by||^2_{L^\infty(0,1;\h)},
\end{equation}
for any $t\ge 1$.

\noindent \texttt{Estimate for $\mathbb{S}_t \left(\frac{\partial \by}{\partial t}\right)$ in $L^2_b(\rp; \ve^\ast)$}

From \eqref{2}, it is clear that
\begin{equation}\label{att-set6}
 \left\lvert\left\lvert\frac{\partial \by}{\partial t}\right \rvert \right\rvert^2_{\ve^\ast}\le C \left(\left\lvert\left\lvert\mathcal{A}\by\right \rvert \right\rvert^2_{\ve^\ast}+
\left\lvert\left\lvert\mathcal{A}_p \by\right \rvert
\right\rvert^2_{\ve^\ast}+
\left\lvert\left\lvert\mb(\by,\by)\right \rvert
\right\rvert^2_{\ve^\ast}+ \left\lvert\left\lvert g_0\right \rvert
\right\rvert^2_{\ve^\ast}\right).
\end{equation}
With the definition of $\mathcal{A}$ (see \eqref{def-A}), the estimate \eqref{est-rem} of Remark \ref{rem-1}, and \eqref{ext-trib5} we derive from \eqref{att-set6} that
\begin{equation}\label{83}
\begin{split}
 \int_{t}^{t+1} \left\lvert\left\lvert\frac{\partial \by(s)}{\partial t}\right \rvert \right\rvert^2_{\ve^\ast}
 ds\le C \sup_{s\in [t,t+1]}|\by(s)|^2\int_t^{t+1}||\by(s)||^2 ds+ C \int_t^{t+1} ||g_0(s)||^2 ds\\
+C\int_t^{t+1}||\by(s) ||^2+C\sup_{s\in [t,t+1]}|\by(s)|^2+C
\end{split}
\end{equation}
where the constants involved in this estimate depend only on the
data of \eqref{*}. We infer from this last estimate,
\eqref{att-set1} and \eqref{att-set5} that there exists positive
constants $C_3$ and $R_3$  such that
\begin{equation}\label{att-set7}
 \left( \int_{t}^{t+1} \left\lvert\left\lvert\frac{\partial \by(s)}{\partial t}\right \rvert \right\rvert^2_{\ve^\ast} ds\right)^\frac{1}{2}\le C_3 ||\by||^2_{L^\infty(0,1;\h)}
+R_3,
\end{equation}
 for any $t\ge 1$.

Now, the proof of the lemma follows from \eqref{att-set1}, \eqref{att-set5} and \eqref{att-set7} if $p\in (1,2]$.

\underline{\textsc{Case II: $p\in (2,\frac{2n+6}{n+2}]$} }

\noindent
We should first not that the estimates \eqref{att-set1} and \eqref{att-set5} hold true for the case $p\le 2$ and $p>2$, the main difference lies in th estimating of
 $\mathbb{S}_t \left(\frac{\partial \by}{\partial t}\right)$.
Throughout this step $\frac{q}{2}=p-1$ which will belong to $(1,2)$ provided that $p\in (2,\frac{2n+6}{n+2}]$.
Because of this fact,  we can derive from H\"older's inequality, \eqref{att-set1} and \eqref{att-set5} that
\begin{align*}
||\mathbb{S}_t\by||^2_{L^2_b(\rp;\h)}\le R_4+C_4 e^{-\frac{q\lambda t}{2}} ||\by||^q_{L^\infty(0,1;\h)},\\
||\mathbb{S}_t\by||^2_{L^2_b(\rp;\ve)}\le R_5+C_5
e^{-\frac{q\lambda t}{2}} ||\by||^q_{L^\infty(0,1;\h)},
\end{align*}
for any $t\ge1$.
 Owing to these last inequalities we infer from \eqref{est-rem2}, \eqref{att-set1} and \eqref{att-set5}  that
\begin{equation*}
 \int_t^{t+1} ||\mathcal{A}_p \by(s)||_{\ve^\ast}^2ds \le R_6 +C_6 ||\by||^{2 \frac{(1-a)q}{2}}_{L^\infty_b(0,T;\h)}||\by||^{2\frac{aq}{2}}_{L^\infty_b(0,T;\h)} e^{-\frac{q\lambda t}{2}},
\end{equation*}
which easily implies that
\begin{equation*}
 \int_t^{t+1} ||\mathcal{A}_p \by(s)||_{\ve^\ast}^2ds \le R_7 +C_5 ||\by||^{q}_{L^\infty_b(0,T;\h)}e^{-\frac{q \lambda t}{2}},
\end{equation*}
 With these estimates we can derive from \eqref{83} that
\begin{equation*}
 \left( \int_{t}^{t+1} \left\lvert\left\lvert\frac{\partial \by(s)}{\partial t}\right \rvert \right\rvert^2_{\ve^\ast} ds\right)^\frac{1}{2}\le C_8 e^{-\frac{q\lambda t}{2}}||\by||^q_{L^\infty(0,1;\h)}
+R_8.
\end{equation*}
Now the proof of \eqref{75} follows from all of these estimates.
\end{proof}
\begin{proof}[\textbf{Proof of Theorem \ref{TRAJ-THM}}]

First of all, Lemma \ref{lem-traj} implies that the symbol space
$\Sigma(g_0)$ is a compact and complete metric space. Secondly, we
see from Lemma \ref{lem-traj2} that the family $\{\mck_g^+, g\in
\Sigma(g_0)\}$ is $(\thl, \Sigma(g_0))$-closed. Moreover, it
follows from Lemma \ref{lem-traj3} that the translation semigroup
$\mathbb{S}_t$ is continuous on $\Sigma(g_0)$, and $
\mathbb{S}_t\Sigma(g_0)\subset \Sigma(g_0)$ for all $t\ge0$,  and
the family $\{\mck_g^+, g\in \Sigma(g_0)\}$ is translation
coordinated (that is, it verifies \eqref{tr-coo} ). Thanks to
Lemma \ref{TRAJ-lem} the ball $B_{R_0}$
\begin{equation*}
B_{R_0}=\{\bu\in \mathcal{F}^a_+: ||\bu||_{\mathcal{F}^a_+}\le
2R_0\},
\end{equation*}
is a uniformly attracting set for $\{\mck_{g}^+, g\in
\Sigma(g_0)\}$ in $\thl$. The set $B_{R_0}$ is compact in $\Theta^{loc}_+$
and it is obviously bounded in $\mathcal{F}^a_+$. Thus the
conditions of \cite[Theorems 3.1]{VISHIK+CHEPYZHOV} can
be applied to our case, and this proves our theorem.
\end{proof}
In the next assertion we will study the structure of the uniform trajectory attractor of Theorem \ref{TRAJ-THM}. To this end we should
 extend all the concepts we have introduced throughout this section to the case where the external force $g_0$ has an extension to $\mathbb{R}$. This is possible thanks to the invariance
property \eqref{inv-ga} and arguing exactly as in \cite[Section 4]{VISHIK+CHEPYZHOV}. Then, we denote by $\mathcal{F}_{loc}$ and $\mathcal{F}^a$ the analog of
$\mfl$ and $\mfa$ by taking $\mathbb{R}$ in place of $\rp$. In a similar way we also define $\Theta^{loc}$. For a function $\bu$ defined on $\mathbb{R}$ we designate by
 $\Pi_+\bu$ its restriction to $\rp$.
\begin{Def}
 The complete symbol space $Z(g_0)$ is the set of all functions $\zeta\in \Ll^2(\mathbb{R};\ve^\ast)$ such that $\Pi_+\zeta(t+\cdot)$ the restriction of
 $\zeta(t+\cdot)$ to $\rp$ belongs to $\omega(\Sigma(g_0))$ for any $t\in \mathbb{R}$.

For any $\zeta\in Z(g_0)$, the kernel $\mathbb{K}_\zeta$ is the
set of all weak solutions $\by$ to \eqref{2} with external force
$\zeta$ such that $\by\in \mathcal{F}_{loc}$ and
  $\by$ satisfies the energy inequality \eqref{NRJ} for almost everywhere $t,\tau\in \mathbb{R}$ with $t\ge \tau+1$.
For a translation bounded function $g_0\in
\Ll^2(\mathbb{R};\ve^\ast)$ we set
$$\mathbb{K}_{Z(g_0)}=\bigcup_{\zeta \in Z(g_0)}\mathbb{K}_\zeta. $$
\end{Def}
 As in \cite{VISHIK+CHEPYZHOV} we also have the following result.
\begin{thm}\label{TRAJ-THM2}
 Under the assumptions of Theorem \ref{TRAJ-THM},
\begin{equation}
\mathcal{A}_{\Sigma(g_0)}=\mathcal{A}_{\omega(\Sigma(g_0))},
=\Pi_{+}\bigcup_{\zeta
 \in Z(g_0)}\mathbb{K}_{\zeta}=\Pi_+ \mathbb{K}_{Z(g_0)}.
 \end{equation}

The kernel $\mathbb{K}_\zeta$ is not empty for any $\zeta \in Z(g_0)$, and the set $\mathbb{K}_{Z(g_0))}$ is bounded in $\mathcal{F}^a$ and compact in
   $\Theta^{loc}_+$.
\end{thm}
\begin{proof}
We omit the proof of this result as it is very similar to the proof of \cite[Theorem 4.1]{VISHIK+CHEPYZHOV}.
\end{proof}
From the construction we made in Section \ref{Prel} that
$\mathcal{A}$ is self-adjoint and thanks to Rellich's theorem it
is compact on $\h$.
 Therefore we can define the fractional power operators $\mathcal{A}^\delta$ with domain $D(\mathcal{A}^\delta)$ for $\delta \in \mathbb{R}$.
For any reflexive Banach space $X$ and $Y$ let $$W_{p_0, p_1}(0,T,
X, Y)=\left\{\bu(s), s\in [0,T]: \bu \in L^{p_0}(0,T;X),
\frac{\partial \bu}{\partial t} \in L^{p_1}(0,T;Y)\right\},$$
which is a Banach space when endowed with the graph norm. Taking
$\delta\in [0,1/2)$ implies that the embedding  $\h\subset
D(\mathcal{A}^{-\delta})$ is compact and
$D(\mathcal{A}^{-\delta})\subset \ve^\ast$ is continuous.
 Therefore, owing to \cite[Corollary 4]{Simon} the space $W_{\infty, 2}(0,T, \h, \ve^\ast)$ is compactly embedded in $C(0,T; D(\mathcal{A}^{-\delta}))$.
Since $\ve\subset D(\mathcal{A}^\delta)$ is compact then by invoking \cite[Lemma 5.1]{LIONS}
we see that $W_{2,2}(0,T, \ve, \ve^\ast)$ is compact in $L^2(0,T; D(\mathcal{A}^\delta))$. The consequence of these facts and Theorem \ref{TRAJ-THM2} are given below.
\begin{cor}
 For any set $\mathbb{B}\subset\mck^+_{\Sigma(g_0)}$ which is bounded in $\mfa$ we have
\begin{align}
 \dist_{C(0,T; D(\mathcal{A}^{-\delta}))}(\Pi_{0,T}\mathbb{S}_T\mathbb{B}, \Pi_{0,T}\mathbb{K}_{Z(g_0)})\rightarrow 0 \text{ as } t\rightarrow \infty,\\
 \dist_{L^2(0,T; D(\mathcal{A}^{\delta}))}(\Pi_{0,T}\mathbb{S}_T\mathbb{B}, \Pi_{0,T}\mathbb{K}_{Z(g_0)})\rightarrow 0 \text{ as } t\rightarrow \infty.
\end{align}
\end{cor}
In the corollary $\Pi_{0,T}\bu$ stands for the restriction to $[0,T]$ of a function $\bu$ defined on $\mathbb{R}$ (or $\rp$) and $\dist(X,Y)$ denotes the distance between the set
$X$ and $Y$.
\section{Acknowledgement}
We wish to thank the anonymous referee for his/her insightful comments which considerably improved this article. 
The author is very grateful to the support he received from the Austrian Science Foundation. 


\begin{thebibliography}{99}

\bibitem{Babin-1995}
A.~V.~Babin. The attractor of a generalized semigroup generated by an elliptic equation in a tube domain.
\newblock{\em Russian Acad. Sci. Izv. Math.} 44(2): 207-223, 1995. 

\bibitem{Babin+Vishik-1985}
A.~V.~Babin and M.~I.~Vishik.
Maximal attractors of semigroups corresponding to evolutionary differential equations. {\em  Mat. Sb.}  126(3): 397-419, 1985.

\bibitem{BABIN+VISHIK}A. V.~Babin and M. I.~Vishik. \newblock{ \em Attractors of evolution equations.}
Volume 25 of \newblock{\em
 Studies in Mathematics and its Applications}. North-Holland Publishing Co., Amsterdam, 1992.

\bibitem{Caraballo+al-2010}
F.~Balibrea, T.~Caraballo, P.~E.~Kloeden and J.~Valero.
Recent developments in dynamical systems: three perspectives. 
\newblock{ \em Internat. J. Bifur. Chaos Appl. Sci. Engrg.} 20(9): 2591-2636, 2010.

\bibitem{Ball-1978} J.~M.~Ball. 
On the asymptotic behavior of generalized processes, with applications to nonlinear evolution equations. 
\newblock{\em J. Differential Equations.} 27(2): 224-265, 1978. 

\bibitem{Barbashin-1948} E.~A.~Barbashin. On the theory of generalized
dynamical systems. \newblock{ \em Moskov. Gos. Ped. Inst. Uchen.
Zap. 135, Mat.} 2: 110-133, 1948.

\bibitem{BELLOUT2} H.~Bellout, F.~Bloom and J.~Necas. %
\newblock{Phenomenological behavior of multipolar viscous fluids.} %
\newblock{\em Quarterly of Applied Mathematics} 50:559-583, 1992.

\bibitem{BELLOUT1} H.~Bellout, F.~Bloom and J.~Necas. \newblock{Solutions
for incompressible Non-Newtonian fluids}. \newblock{\em C. R. Acad. Sci.
Paris S\'er I. Math.} 317:795-800,1993.

\bibitem{BELLOUT4} H.~Bellout, F.~Bloom and J.~Necas. \newblock{Young
measure-valued solutions for Non-Newtonian incompressible fluids.} %
\newblock{Communication in Partial Differential Equations.} 19(11\&
12):1763-1803, 1994.

\bibitem{BELLOUT3} H.~Bellout, F.~Bloom and J.~Necas. \newblock{Bounds for
the dimensions of the attractors of nonlinear bipolar viscous fluids}. %
\newblock{\em Asymptotic Analysis.} 11(2):131-167,1995.

\bibitem{BELLOUT5} H.~Bellout, F.~Bloom and J.~Necas. \newblock{Existence,
uniqueness and stability of solutions to initial boundary value problems for
bipolar fluids.}  \newblock{\em Differential and Integral Equations}
8:453-464, 1995.

\bibitem{BISKAMP}D.~Biskamp. \newblock{\em Magnetohydrodynamical Turbulence}. Cambridge University Press, Cambridge, 2003.

\bibitem{Caraballo+al-2003a}T.~Caraballo, P.~Mar\'in-Rubio and J.~C.~Robinson. A 
comparison between two theories for multi-valued semiflows and their asymptotic behaviour.
\newblock{ \em Set-Valued Anal.} 11(3): 297-322, 2003.

\bibitem{Caraballo+al-2003b} T.~Caraballo, J.~A.~Langa, V.~S.~Melnik and J.~Valero. 
Pullback attractors of nonautonomous and stochastic multivalued dynamical systems.
\newblock{\em Set-Valued Anal.} 11(2): 153-201, 2003. 

\bibitem{Chandrasekhar} S.~Chandrasekhar.
\newblock{\em Hydrodynamic and
Hydromagnetic Stability.} Dover, 1981.

\bibitem{VISHIK+CHEPYZHOV-94}V. V.~Chepyzhov and M. I.~Vishik.
\newblock{Attractors of non-autonomoous dynamical systems and their dimension}.
\newblock{\em J. Math. Pures Appl.(9)}   {73}: 279-333, 1994.

\bibitem{VISHIK+CHEPYZHOV}V. V.~Chepyzhov and M. I.~Vishik.
\newblock{Evolution equations and their trajectory attractors}.
\newblock{\em J. Math. Pures Appl.(9)}   {76}(10): 913-964, 1997.

\bibitem{VISHIK+CHEPYZHOV-2002}V. V.~Chepyzhov and M. I.~Vishik. \newblock{Trajectory and global
attractors of the three-dimensional Navier-Stokes system.}
\newblock{\em Math. Notes}  71(1-2):177-193, 2002.


\bibitem{VISHIK+CHEPYZHOV+WENDLAND}V. V.~Chepyzhov, M. I.~Vishik, and W.
L.~Wendland. \newblock{On non-autonomous sine-Gordon type
equations with a simple global attractor and some averaging.}
 \newblock{\em Discrete Contin. Dyn. Syst.} 12(1):27-38, 2005.

\bibitem{Crauel}H.~Crauel and F.~Flandoli. Attractors for random dynamical
systems. \newblock{\em Prob. Theor. Related Fields.} 100: 365-393,
1994.

\bibitem{Craueletal}H.~Crauel, A.~Debussche, F.~Flandoli. \newblock{Random attractors}.
\newblock{\em J. Dynam. Differential Equations} 9(2):307-341, 1995.



\bibitem{DESJARDINS}B.~Desjardins, and C.~Le Bris. Remarks on a nonhomogeneous model of magnetohydrodynamics.
\newblock{\em Differential Integral Equations.} 11(3):377-394, 1998.


\bibitem{DU} Q.~Du and M.D.~Gunzburger. \newblock{Analysis of Ladyzhenskaya
Model for Incompressible Viscous Flow.}  \newblock{\em Journal of
Mathematical Analysis and Applications.} \textbf{155}:21-45, 1991.

\bibitem{DUVAUT+LIONS}G.~Duvaut and J.-L.~Lions. \newblock{In\'equations en thermo\'elasticit\'e et
magn\'etohydrodynamique.} \newblock{\em Arch. Rational Mech. Anal.}
\textbf{46}: 241-279, 1972.

\bibitem{RUZICKA} J.~Freshe and M.~Ruzicka. \newblock{Non-homogeneous
generalized Newtonian fluids.} \newblock{Mathematische
Zeitschrift}. \textbf{260}(2):353-375, 2008.


\bibitem{LEBRIS}J.-F.~Gerbeau, and C.~Le Bris. Existence of solution for a density-dependent magnetohydrodynamic equation.
\emph{Adv. Differential Equations} 2(3):427-452, 1997.

\bibitem{GERBEAU+LEBRIS}J.-F.~Gerbeau, C.~Le Bris, and T.~ Leli\`evre. \newblock{\em Mathematical methods for the Magnetohydrodynamics of Liquid Metals.}
 Oxford University Press, New York, 2006.


\bibitem{GUNZBURGER} M.D.~Gunzburger, O.A.~Ladyzhenskaya, and J.S.~Peterson.
 On the global unique solvability of initial-boundary value problems for the coupled modified Navier-Stokes and Maxwell equations.
 \emph{J. Math. Fluid Mech.} 6(4):462-482, 2004.

\bibitem{GUNZBURGER2}M.D.~Gunzburger and C.~Trenchea. Analysis of an optimal control problem for the three-dimensional coupled modified Navier-Stokes and Maxwell equations.
\emph{J. Math. Anal. Appl.} 333(1):295-310, 2007.

\bibitem{Haraux} A.~Haraux. \newblock{\em Syst\`ems Dynamiques Dissipatifs et Applications.} Masson, Paris, 1991.


\bibitem{Kapustyan}A. V.~Kapustyan. Global attractors of a nonautonomous
reaction-diffusion equation. \newblock{\em Differ. Equ.}
38(10):1467-1471, 2002.

\bibitem{Kapustyan-2010} O.~V.~Kapustyan and J.~Valero. 
Comparison between trajectory and global attractors for evolution systems without uniqueness of solutions. 
\newblock{\em Internat. J. Bifur. Chaos Appl. Sci. Engrg.} 20(9): 2723-2734, 2010. 


\bibitem{LADY1} O.A.~Ladyzhensakya. \newblock{\em The mathematical theory of
viscous incompressible flow.} Gordon and Breach, New York, 1969.

\bibitem{LADY2} O.A.~Ladyzhensakya. \newblock{New equations for the
description of the viscous incompressible fluids and solvability in the
large of the boundary value problems for them.}  In \newblock{\em Boundary
Value Problems of Mathematical Physics V.} American Mathematical Society,
Providence, RI, 1970.


\bibitem{LADY-SOLO}O.A.~ Ladyzhenskaya and V.~Solonnikov. Solution of some nonstationary magnetohydrodynamical
problems for incompressible fluid. \emph{Trudy of Steklov Math. Inst.} 69:115-173, 1960.

\bibitem{LIONS} J.-L.~Lions.  \emph{Quelques m\'ethodes de 
r\'esolution des probl\`emes aux limites non lin\'eaires.} Dunod; Gauthier-Villars, Paris, 1969.

\bibitem{Malek+Necas-1996} J.~M\'alek and J.~Necas. 
A finite-dimensional attractor for three-dimensional flow of incompressible fluids. 
\newblock{\em J. Differential Equations.} 127(2): 498-518, 1996.

\bibitem{MALEK} J.~Malek, J.~Necas and A.~Novotny. \newblock{Measure-valued
solutions and asymptotic behavior of a multipolar model of a boundary layer.}
\newblock{\em Czechoslovak Mathematical Journal.} 42(3):549-576, 1992.

\bibitem{MALEKetal}J.~M\'alek, J.~Necas, M.~Rokyta and M. Ruzicka.
\newblock{\em Weak and measure-valued solutions to evolutionary PDEs.} Applied Mathematics and Mathematical Computation, 13.
 Chapman \& Hall, London, 1996.

\bibitem{Melnik+Valero-1998} V.~S.~Melnik and J.~Valero. On attractors of multivalued semi-flows and differential inclusions.
\newblock{\em Set-Valued Anal.} 6(1): 83-111, 1998. 

\bibitem{NECAS2} J.~Necas and M.~Silhavy \newblock{Multipolar viscous
fluids.} \newblock{\em Quaterly of Applied Mathematics.} XLIX(2):247-266,
1991.

\bibitem{NECAS1} J.~Necas, A.~Novotny and M.~Silhavy. \newblock{Global
solution to the compressible isothermal multipolar fluids}. \newblock{\em J.
Math. Anal. Appl.} 162:223-242, 1991.

\bibitem{ROBINSON}J. C.~Robinson. \newblock{\em Infinite-dimensional dynamical systems. An introduction to dissipative parabolic
 PDEs and the theory of global attractors.
Cambridge Texts in Applied Mathematics.} Cambridge University
Press, Cambridge, 2001.

\bibitem{SAMOKHIN3} Samokhin, V. N. On a system of equations in the magnetohydrodynamics of nonlinearly viscous media.
\emph{Differential Equations} 27(5):628-636, 1991.

\bibitem{SAMOKHIN4}Samokhin, V. N. Existence of a solution of a modification of a system of equations of magnetohydrodynamics.
\emph{Math. USSR-Sb}. 72(2):373-385, 1992.

\bibitem{SAMOKHIN2}V. N.~Samokhin. Stationary problems of the magnetohydrodynamics of non-Newtonian media.
\emph{Siberian Math. J.} 33(4):654-662, 1993.

\bibitem{SAMOKHIN}V. N.~Samokhin. The operator form and solvability of equations of the magnetohydrodynamics of nonlinearly viscous media.
  \emph{Differ. Equ.} 36(6):904-910, 2000.

\bibitem{Schmalfuss}B.~Schmalfu\ss. Backward cocycles and attractors of stochastic
differential equations, in \newblock{\em International Seminar on
Applied Mathematics-Nonlinear Dynamics: Attractor Approximation
and Global Behaviour.} V. Reitmann, T. Riedrich, and N. Koksch
(eds.), TU Dresden, 1992, pp 185-192.

\bibitem{Sell-1996}G.~R.~Sell. Global attractors for the three-dimensional Navier-Stokes equations. \newblock{\em J. Dyn. Diff.
Eqs.} 8(1): 1-33, 1996.




\bibitem{TEMAM+SERMANGE}M.~Sermange and R.~Temam. \newblock{Some mathematical questions related to the MHD equations.}
\newblock{\em  Comm. Pure Appl. Math.} \textbf{36}(5): 635-664, 1983.


\bibitem{Simon} J. Simon. Compact sets in the space $L^{p}(0;T;B)$, Annali
Mat. Pura Appl. 146, IV, 65-96, 1987.



\bibitem{STUPELIS} L.~Stupyalis. An initial-boundary value problem for a system of equations of magnetohydrodynamics. 
\newblock{\em Lithuanian Math. J.} 40(2):176-196, 2000.

\bibitem{TEMAM-INF}R.~Temam. \newblock{\em Infinite-dimensional dynamical systems in
mechanics and physics.} Second edition. Volume 68 of \newblock{\em
Applied Mathematical Sciences}. Springer-Verlag, New York, 1997.

 \bibitem{Temam} R.~Temam. \newblock {\em Navier-{S}tokes {E}quations}. %
\newblock North-Holland, 1979.

\bibitem{VISHIK}M.I.~Vishik. \newblock{\em Asymptotic Behaviour of Solutions of Evolutionary Equations}. Cambridge University Press,
Cambridge, 1992.


\bibitem{Zeidler-2A}E. Zeidler.
\newblock{\em Nonlinear Functional Analysis and its Applications, II/A:
Linear Monotone Operators.} Springer-Verlag, New York, 1990.

\bibitem{ZHOU}C.~Zhao, S.~Zhou and Y.~Li. \newblock{Trajectory attractor and global attractor for a two-dimensional
incompressible non-Newtonian fluid.} \newblock{\em
 J. Math. Anal. Appl.}  325(2): 1350-1362, 2007.
\end{thebibliography}
\end{document}